\documentclass[a4paper, 11pt]{article}
\usepackage{graphicx} 
\usepackage{amsmath, relsize} 
\usepackage{todonotes} 
\usepackage{amsthm} 
\usepackage{amsfonts}
\usepackage{bbm}
\usepackage{amssymb}
\usepackage{array}
\usepackage{url}
\usepackage{amsmath}
\usepackage{amssymb}
\usepackage{amsthm}
\usepackage{bbm}

\usepackage{enumerate}
\usepackage{setspace}
\usepackage{graphicx,color,hyperref}
\usepackage[margin=0.885in]{geometry}
\parskip \medskipamount
\newtheorem{theorem}{Theorem}[section]
\newtheorem{proposition}[theorem]{Proposition}

\newtheorem{lemma}[theorem]{Lemma}
\theoremstyle{definition}
\newtheorem{definition}[theorem]{Definition}
\newtheorem{remark}[theorem]{Remark}

\numberwithin{equation}{section}

\begin{document}

\title{Essential enhancements in Abelian networks: \\ continuity and uniform strict monotonicity}
\author{Lorenzo Taggi \\
{ \small \textit{Sapienza Universit\`a di Roma} } \\
{\small \textit{Dipartimento di Matematica `Guido Castelnuovo'}}}
\maketitle

\begin{abstract}
We prove  that  in wide generality the critical curve of the activated random walk model  is a continuous function of the deactivation rate, and  we provide a  bound on its slope which is uniform with respect to the choice of the graph. Moreover, we derive strict monotonicity properties for the probability of a wide class of  `increasing' events,
extending previous results of Rolla and Sidoravicius (2012).  Our proof method is of independent interest and can be viewed as a reformulation of the `essential enhancements'  technique -- which was introduced for percolation -- in the framework of Abelian networks.
\end{abstract}


\section{Introduction}
\label{sect:Intro}
The activated random walk model (ARW) is a particle system with conserved number of particles. It is a special case of a class of models introduced by Spitzer in the '70s and it is not only of great mathematical interest but also physically  relevant due to its connections to  \textit{self-organised criticality}  \cite{Dickman}. The informal definition  of the model is as follows. Let $G = (V, E)$ be a infinite  undirected unimodular graph (for example $\mathbb{Z}^d$ or a regular tree). Each particle can either be of type A (active) or of type S (sleeping, or inactive). At time zero, the number of particles is sampled according to a Poisson distribution with parameter  $\mu \in [0, \infty)$ independently at every vertex, where $\mu$ is the \textit{particle density}, and every particle is of type A. 
An independent exponential clock with rate $\lambda \in [0, \infty)$, the \textit{deactivation rate,} is associated to every active particle. 
Every A-particle performs a continuous time simple random walk independently until its own clock rings.
When this happens, the A-particle turns into the S-state. Every S-particle is at rest.  Moreover, whenever a  S-particle shares the vertex with an A-particle, the S-particle is instantaneously activated, i.e, it becomes an A-particle. It follows from this definition that, almost surely, a particle of type S can be observed only if it does not share the vertex with other particles.

Let $\mathbb{P}_{\lambda, \mu}$ be the probability measure of the interacting particle system defined informally above, whose existence on unimodular graphs was proved in \cite{Rolla2}. A central and  natural question is whether the dynamics dies out with time or whether it is sustained at all times. More precisely, we say that the system \textit{fixates} if for every finite set  $A \subset V$ there exists a time $t_A < \infty$ such that for any time $t > t_A$ no active particle jumps from a vertex of $A$, and that it is \textit{active} if it does not fixate. The \textit{critical  density} is defined as,
\begin{equation}\label{eq:criticaldensity}
\forall \lambda \in [0, \infty), \quad \mu_c(\lambda) := 
\inf \big \{ \mu \in \mathbb{R}^+_{0} \, \, : \, \, 
\mathbb{P}_{\lambda, \mu} ( \,  \mbox{ARW is active }  \,  ) > 0  \big \}.
\end{equation}
It was proved in  \cite{Rolla} that  the probability
that the model is active 
is either zero or one, that it does not decrease with $\mu$ and does not increase with $\lambda$.
This ensures
the existence of a unique transition point between the regime of  a.s. local fixation and the regime of a.s. activity.
In recent years significant effort has been made for proving basic properties of the critical curve, 
$\mu = \mu_c(\lambda)$. 
It was proved in \cite{Stauffer} that $\mu_c(\lambda) \geq \frac{\lambda}{1+\lambda}$ in any vertex-transitive graph, generalising and extending previous results from \cite{Rolla, Sidoravicius}.
It is known from \cite{Shellef} that $\mu_c(\lambda) \leq 1$ for any $\lambda \in [0, \infty)$ in wide generality.  It was proved  in \cite{Forien, Rolla2,  Stauffer,  Taggi, Taggi2} that on various graphs $\mu_c(\lambda) < 1$ for any $\lambda \in (0, \infty)$  and that $\mu_c(\lambda) \rightarrow 0$ as  $\lambda \rightarrow 0$.
Moreover, it was proved in  \cite{Asselah, Basu} that, on $\mathbb{Z}$,
 $\mu_c(\lambda) = O(  \sqrt{ \lambda } )$ in the limit as $\lambda \rightarrow 0$.
It was proved in \cite{Rolla3} that the critical density is universal. Our first main theorem states a new general property of the critical curve, namely that it is a continuous function of the deactivation parameter $\lambda$.

\begin{theorem}\label{theo:continuity}
On any unimodular  graph the two following properties hold:
\begin{enumerate}
\item $\mu_c(\lambda)$  is a continuous function of $\lambda$ in  $(0, \infty)$,
\item 
for any $\lambda \in (0, \infty)$,
$\limsup\limits_{ \delta \rightarrow 0} \, \,  
\frac{\mu_c( \lambda + \delta) - \mu_c(\lambda)}{\delta} \leq \frac{1}{\lambda (1 + \lambda)}.$
\end{enumerate}
\end{theorem}
Continuity of the critical curve at $\lambda = 0$ (more precisely, right-continuity, namely $\lim_{ \lambda \rightarrow 0^+  } \mu_c(\lambda) = \mu_c(0) = 0$)
in $\mathbb{Z}^d$ was proved in
\cite{Basu, Forien, Stauffer}.
 Our Theorem \ref{theo:continuity} generalizes such a continuity property to all positive values of $\lambda$ and holds for any  unimodular  graph. 
Even though the critical curve is expected to strongly depend on the graph, the  second claim of Theorem \ref{theo:continuity} provides a bound on its slope  which is uniform with respect to the choice of the graph.
Moreover,  the assumption that the graph is unimodular  is only required to give sense to the continuous time dynamics and then to (\ref{eq:criticaldensity}).
  For a more general notion of critical density  (see equation (\ref{eq:generaldefinitiondensity})  and Remark \ref{rem:generality} below) our theorem holds on any  locally-finite  infinite connected graph.

\subsection{Strict monotonicity properties}
Our first main theorem is a consequence of our second  theorem, which derives  new general monotonicity properties for the probability of a wide  class of  events, 
which will be referred to as `relevant'. 
This  class includes  all the events which are increasing and which depend
on how many times the vertices are visited by the particles
(we refer to Section \ref{sect:relations} for a precise definition). 
For example, the event $\mathcal{A} = \{  \forall x \in K, M(x) > H(x) \}$,
where $K \subset V$ is finite, $M(x)$ is the number of times the active particles jump from $x$ and 
$ ( H(x)   )_{x \in K}$ is any integer-valued vector,  is relevant.
 This is an important class of events, since one can deduce 
whether the system fixates or is active by  determining the limiting probability
of appropriately defined sequences of these events. 

The derivation of  monotonicity properties is very useful and allows a deeper understanding of the model. 
From the definition of the activated random walk dynamics it is reasonable to expect that 
the probability of any relevant event is \textit{non-increasing} with respect to $\lambda$ and \textit{non-decreasing} with respect to $\mu$.
The proof of this claim is non-trivial and was derived in \cite{Rolla} by employing a graphical representation. 

Here we address a related question, namely do monotonicity properties hold if we increase the deactivation rate and the particle density \textit{at the same time}? This question is challenging,  since  the increase of the deactivation rate and  of the particle density play against each other.  Indeed,  higher deactivation rate implies that the model is `less active', while higher particle density implies that the model is `more active'. 
Our Theorem \ref{theo:strict monotonicity} below 
studies a regime where a positive increase in $\mu$ compensates for a small enough increase in $\lambda$.  
More precisely,   if take an arbitrary point of the phase diagram, 
$(\lambda, \mu)  \in \mathbb{R}_+^2$, and we  move up-right along a semi-line line which starts from $(\lambda, \mu)$  and  whose slope, $s$, satisfies
$ s \geq \frac{1}{\lambda(1 + \lambda)}$, then  the probability of the  event does not decrease.
Remarkably, our estimate on the minimal slope is uniform not only with respect to the choice of the graph, but also with respect to the choice of the event, provided that it is relevant.  The monotonicity result of Rolla and Sidoravicius \cite{Rolla} can thus be viewed as corresponding to the special case $s = \infty$ of our theorem.
\begin{theorem} \label{theo:strict monotonicity}
Consider any unimodular graph,  let  $\mathcal{A}$ be any relevant event.  Let $(\lambda, \mu) \in  \mathbb{R}_+^2$ be an arbitrary point of the phase diagram, let 
$
\mathcal{C}_{\lambda, \mu}
$
be the region above the semi-line with slope  $\frac{1}{\lambda(1+\lambda)}$ which starts from $(\lambda, \mu)$,
\begin{equation}\label{eq:curve}
\mathcal{C}_{\lambda, \mu} : = \Big \{
(x, y) \in \mathbb{R}^2 \, \, : \, \, y \geq  \frac{1}{\lambda(1 + \lambda)} \, \, (\, x \, - \, \lambda) \,  + \, \mu, \,  \, x \geq \lambda \, 
\Big \} .
\end{equation}
Then, 
for any pair $(\lambda^\prime, \mu^\prime) \in \mathcal{C}_{\lambda, \mu}$,
\begin{equation}\label{eq:statementmonotonicity}
 \mathbb{P}_{  \lambda, \, \mu  } ( \mathcal{A})
 \leq \mathbb{P}_{  \lambda^\prime,\,  \mu^\prime  } ( \mathcal{A}).
\end{equation}
\end{theorem}
As we show in Section \ref{sect:relations} below, relevant events can be defined 
in the framework of the Diaconis-Fulton representation, which is well-defined on any locally-finite graph.  Hence our theorem can be stated in wider generality, see Remark \ref{rem:generality} below.

\subsection{Proof method: Essential enhancements}
Our proof method can be viewed as a reformulation  of the `Essential enhancements' 
technique -- which was mostly used in Percolation \cite{Aizenman, Ballister} -- in the framework of Abelian networks.
Our method may find applications in the study of other  Abelian models, for example the frog model \cite{Hoffman}, 
oil and water \cite{Candellero1, Candellero2},  the stochastic sandpile model \cite{Rolla}, the Abelian sanpdiles  \cite{Jarai}, see also \cite{Bond}.
Our proof uses the setting of the Diaconis-Fulton graphical representation \cite{Rolla}, where some random instructions -- operators which act on the  particle configuration moving active particles to their neighbours or trying to let the A-particle turn into a S-particle -- are used to mimic the dynamics without employing the variable `time'. Such a graphical representation fulfils the fundamental Abelian property which, informally, states that the relevant quantities -- for example the number of times an active particle jumps from each vertex  -- do not depend on the order according to which such instructions are used. 

Our proof is divided into three main steps.
The \textit{first step} of the proof is the derivation of a Russo's formula \cite{Russo} -- which is a classical formula in percolation -- for  activated random walks, Theorem \ref{theo:Russo} below. This formula  relates the partial derivative with respect to $\lambda$ of the probability of  any relevant event  to the expected number of instructions which are `sleeping essential'  for the event. Such instructions will be defined later and, informally, are those instructions whose removal would cause the occurrence of the event. Similarly, such a formula  relates the partial derivative with respect to $\mu$ of the probability of any relevant  event to the expected number of vertices which are `particle essential' for the event, namely vertices such that the addition of one more particle there would cause the occurrence of the event.
In the \textit{second step} of the proof we derive the following differential inequality, which holds for any  relevant event $\mathcal{A}$,
\begin{equation}\label{eq:inequality}
- \frac{\partial }{\partial \lambda } \mathcal{P}_{\lambda, \mu  }( \mathcal{A} ) \leq \frac{1}{\lambda(1 + \lambda)} 
\frac{\partial }{\partial \mu } \mathcal{P}_{\lambda, \mu  }( \mathcal{A} ),
\end{equation}
where $\mathcal{P}_{\lambda, \mu  }$ is the law of the initial particle configuration and of the random instructions. The two following properties of the \textit{odometer} -- a fundamental quantity which counts how many times an active particle jumps from each  vertex -- are derived and used  for the proof of (\ref{eq:inequality}). The first property is that the removal of  a `sleep' instruction does not affect the value of the odometer,  unless such a removed instruction occupies a very specific location in the array of instructions. Such a property allows us to the deduce that, on each vertex, at most one instruction is `sleeping essential'. The second property states that if the removal of a sleep instruction lets the event $\mathcal{A}$ occur, then also the addition of a particle at the same vertex lets  $\mathcal{A}$ occur, provided that $\mathcal{A}$ is relevant.  This leads to the conclusion that if on a vertex we have a sleeping-essential instruction, then the vertex is also particle-essential. 
Such two properties combined allow the comparison between the partial derivatives and lead to (\ref{eq:inequality}).
In the \textit{third step} we derive our monotonicity theorem by using the differential inequality, (\ref{eq:inequality}), and we derive our main continuity  theorem by using our monotonicity theorem.

We conclude with some natural questions which might be answered by further developing our framework. To begin, the derivation of the inverse of the inequality (\ref{eq:inequality}) (with some other     positive and bounded  constant uniformly  in $\mathcal{A}$ in place of $\frac{1}{\lambda(1+\lambda)}$) would allow us to answer the following open question.

\vspace{0.3cm}

\textit{\textbf{Open Problem 1.}} \textit{Prove that $\mu_c(\lambda)$ is {strictly increasing} with respect to $\lambda$.} 

\vspace{0.3cm}
Our proof shows that the critical density is a continuous function of the deactivation rate and provides a bound for its right and left derivatives, but unfortunately it does not  show that  the right and left derivatives coincide.  This considerations lead to the following natural question, to which we expect the answer to be positive. 

\vspace{0.3cm}

\textit{\textbf{Open Problem 2.}}
\textit{Prove that  $\mu_c(\lambda)$ is differentiable with respect to $\lambda$. }

\vspace{0.3cm}
In the framework of percolation   differentiability properties of several quantities of interest have been studied for example in \cite{Russo2}.



\vspace{0.3cm}

\textit{{Organisation of the paper.}}
This paper is organised as follows. 
In Section \ref{sec:Diaconis} we recall the properties of the Diaconis-Fulton representation.
 In Section
\ref{sect:relations} we introduce the main definitions and discuss the properties of the jump odometer. 
In Section \ref{sect:Russo}
we present  the equivalent of Russo's formula for activated random walk. 
In Section \ref{sect:proofs}  we present the proof of (\ref{eq:inequality}). In Section \ref{sect:theoremsproof} we  present the proof of our main theorems, Theorem
 \ref{theo:continuity}  and  \ref{theo:strict monotonicity}.

\section*{Notation}
We use the notation $\mathbb{N} = \{1, 2,  3, \ldots\}$, $\mathbb{N}_0 = \{0, 1, 2, \ldots \}$,   $\mathbb{R}_+ = \{ x \in \mathbb{R} \, : \, x > 0  \}$, and
  $\mathbb{R}^+_0 = \{ x \in \mathbb{R} \, : \, x \geq 0  \}$
  and the convention $\inf\{ \emptyset \} = \infty$.  
  The following table presents part of the notation which is introduced in Sections \ref{sec:Diaconis} and \ref{sect:relations} below.
  
  
\begin{center}
	\begin{tabular}{ l l }
	$\eta  = (\eta(x))_{x \in V}$ & particle configuration \\
	$\tau = (\tau^{x,j})_{x \in V, j \in \mathbb{N}_0}$ & array of instructions  \\
		 $\mathcal{H} \times \mathcal{I}$ & set of realisations, with $\eta \in \mathcal{H}$ and $\tau \in \mathcal{I}$\\
		 		   $\mathcal{S} \subset \mathcal{H} \times \mathcal{I}$ & smallest $\sigma$-algebra generated by open subsets  of $\mathcal{H} \times \mathcal{I}$  \\
		  $\mathcal{P}_{\lambda, \mu}$ & probability measure on $\mathcal{H} \times \mathcal{I}$ \\
		 	 $\mathcal{H}_a = \mathbb{N}_0^V$ & set of particle configurations with only active particles \\
		$\tau_{x\rho}$, $\tau_{xy}$ & sleep instruction at $x$,  instruction `jump from $x$ to $y$' \\
				$J^{x,\ell}_{\tau}$, with $\ell \geq 1$ &  $\ell$th jump instruction of $\tau$ at $x$, \\	
				$m_{K, \eta, \tau}$,  	$M_{K, \eta, \tau}$ & odometer, jump odometer \\
		 	  $\mathcal{W} \subset \mathcal{H} \times \mathcal{I}$ & 
		    $\mathcal\{ (\eta, \tau) \in \mathcal{H}  :  m_{K, \eta, \tau}(x) < \infty$ for any finite $K \subset V$ and  $x \in K \}$  \\
				$S^{x,\ell}_{\tau}$ & number of s. instr.  between the $\ell-1$th and the $\ell$th j. instr.  \\		
	$\Gamma_-^{x,k}(\tau)$ & array with no s.  instr.  between the $k-1$th and the $k$th j. instr.  \\
		$\Gamma_1^{x,k}(\tau)$ & array with one s.  instr.  between the $k-1$th and the $k$th j.  instr. \\
$\nu_j ,  \,   \nu_{> j} $ & probability that a vertex
 hosts $j$ (resp.  $>  j$) particles  \\
$\nu^{\prime}_{ >j} $  & derivative of $\nu_{>j}= \nu_{>j}(\mu)$  \\
	\end{tabular} 
\end{center}

\section{Definitions and graphical representation}
\label{sec:Diaconis}
In this section we  introduce the Diaconis-Fulton graphical representation for the dynamics of ARW,  partially following~\cite{Rolla}.

\subsection{Particle configuration and array of instructions}
To begin, we fix a graph $G=(V,E)$, which is always assumed to be {undirected},
  {connected},  {infinite},  and  {locally finite}. For any $x \in V$, we denote by $d_x$ the {degree of the vertex $x$}, which corresponds to the number of vertices  which are connected to $x$ by an edge. 
We refer to the arbitrary chosen vertex $o \in V$ as  \textit{root}. 
We write $x \sim y$ when $x$ and $y$ are neighbours, i.e, $\{x,y\} \in E$.
The set of particle configurations is denoted by $\mathcal{H}=\{0,\rho,1,2,3,\ldots\}^V$, where a vertex being in state $\rho$ denotes that the vertex has  one S-particle, while being in state $i\in\{0,1,2,\ldots\}$ denotes that the vertex contains $i$ A-particles. We employ the following order on the states of a vertex: $0 < \rho < 1<2<\cdots$. In a configuration $\eta\in \mathcal{H}$, a vertex $x \in V$ is called \textit{stable} if $\eta(x) \in \{0, \rho \}$, and it is called \textit{unstable} if $\eta(x) \geq 1$.
We denote by $\mathcal{I}$ the \textit{set of arrays of instructions}, i.e, each element of $\mathcal{I}$ is an  array of instructions $\tau = \big ( \tau^{x,j} \big )_{x \in V,  j \in \mathbb{N}_0}$, where for each $x \in V$ and $j \in \mathbb{N}_0$,
$$
\tau^{x,j} \in \{ \tau_{x\rho}  \} \cup \{ \tau_{xy} \, : \,  y \sim x \},
$$
where  $\tau_{xy}$ and  $\tau_{x\rho}$, 
called \textit{jump} and \textit{sleep} \textit{instruction} respectively, are operators acting on the  particle configuration which are defined as follows. Given any configuration $\eta$ such that $x$ is unstable, performing the instruction $\tau_{xy}$ in $\eta$ yields another configuration $\eta'$ such that $\eta'(z)=\eta(z)$ for all $z\in V\setminus\{x,y\}$, $\eta'(x)=\eta(x)- \mathbbm{1}\{\eta(x)\geq 1\}$, and $\eta'(y)=\eta(y)+\mathbbm{1}\{\eta(x)\geq 1\}$. We use the convention that $1+\rho=2$,
while $k - 1$ is defined only if $k \geq 1$.
Similarly, performing the instruction $\tau_{x\rho}$ to $\eta$ yields a configuration $\eta'$ such that 
$\eta'(z)=\eta(z)$ for all $z\in V\setminus \{x\}$, and if $\eta(x)=1$ we have $\eta'(x)=\rho$, otherwise $\eta'(x)=\eta(x)$.
Note that $\tau_{x \rho}$ and $\tau_{xy}$ cannot be applied if $\eta(x) \in \{0, \rho\}$.

\subsection{Use of the instructions and stabilisation of a set}
Fix  a particle configuration $\eta \in \mathcal{H}$ and an instruction array $\tau \in \mathcal{I}$.
We say that the instruction $\tau^{x,j}$ is \textit{legal} for $\eta$ if $x$ is unstable in $\eta$,
otherwise it is \textit{illegal}.
We say that we \textit{use the instruction} $(x,j)$, $x \in V$, $j \in \mathbb{N}$, of the array $\tau$ for $\eta$, or that we use the instruction  $\tau^{x,j}$ for $\eta$, 
when  we act on the current
particle configuration $\eta$ through the operator $\tau^{x,j}$.
When we use an instruction $(x,j)$ for some $j \in \mathbb{N}$, sometimes we may simply say that  `we topple $x$'.
Let $\alpha$ be a sequence 
 $$\alpha = \big ( (x_1, n_1), (x_2, n_2), \ldots, (x_k, n_k) \big ),$$
 define the operator $\Phi_{\alpha, \tau}$ as 
 $$
 \Phi_{\alpha, \tau} : = \tau^{x_k,n_k} \, 
 \ldots \,  \tau^{x_2,n_2} \, \tau^{x_1,n_1},
 $$
 and for $1 \leq \ell \leq k$  define the subsequence $\alpha^{(\ell)} := 
 \big ( (x_1, n_1), (x_2, n_2), \ldots, (x_\ell, n_\ell) \big ).
 $
 We say that $\alpha$ is a \textit{legal sequence} for $\eta$
 if the three following properties hold at the same time:
\begin{enumerate}
\item[(i)] For any $x \in V$, let $u_x := 
\inf \{\ell \in \{1, \ldots, k \} : x_\ell = x\}$.
If $u_x < \infty$, then $n_{u_x}= 0$. In other words, the first instruction which is used at any vertex $x$  is $\tau^{x,0}$.
\item[(ii)] For any $i \in \{1, \ldots  k-1\}$, let $j(i) := \inf \{\ell > i : x_\ell = x_i\}$. If $j(i) < \infty$, then $n_{j(i)} = n_{i} + 1$. In other words, every time we use an instruction,  we use the one which is located `right above' the one which was used right before at the same vertex. 
\item[(iii)]
For any $i \in \{1, \ldots, k\}$, 
$\tau^{x_i, n_i}$ is legal for  
$\eta_{i-1} : =\Phi_{\alpha^{(i-1)}, \tau} \,  \eta.$
\end{enumerate}
Let $m_{\alpha} =  ( m_{\alpha}(x) \, \, : \, \, x \in V  )$
be given by
 $m_{\alpha}(x) \, = \, \sum_{i \in \{1, \ldots, k\}} \mathbbm{1}\{{x_i = x}\},$
the number of times the vertex $x$ appears in $\alpha$.
Let $M_{\alpha, \tau} =  ( M_{\alpha, \tau}(x) \, \, : x \in V  )$
be given by,
 $$M_{\alpha, \tau}(x) \,  = \, \sum_{i=1}^k \mathbbm{1}_{\{{x_i = x, \tau^{x_i,n_i} \neq \tau_{x_i\rho}}\}},$$
the number of jump instructions of $\alpha$.
Let $K$ be a finite subset of $V$. 
A configuration $\eta$ is said to be \textit{stable} in $K$
if all the vertices $x \in K$ are stable. 
We say that $\alpha$ \textit{is contained in $K$}
if $x_i \in K$ for any $i \in \{1, \ldots, k\}$. 
We say that $\alpha$ \textit{stabilizes} $\eta$ in $K$
if every $x \in K$ is stable in $\Phi_{\alpha} \eta$.


\subsection{{Odometers and Abelian property}}
For any subset $K \subset V$, any $x\in V$, any particle configuration $\eta$, and any array of instructions $\tau$, we define 
$$
m_{K,\eta,\tau}(x) : = \sup_{ \alpha \subset K} m_\alpha,
 \quad \quad 
M_{K, \eta,\tau}(x) := \sup_{ \alpha \subset K} M_{\alpha, \tau}, 
$$
where the sup is taken over the legal sequences 
of instructions which are contained in $K$. 
We refer to $m_{K,\eta,\tau}$ as the \textit{odometer function,} or simply \textit{odometer},
and to  $M_{K, \eta,\tau}$ as \textit{jump odometer.}
The following lemma gives a fundamental property of the Diaconis-Fulton representation. 
For the proof we refer to \cite{RollaNotes}.

\begin{lemma}[Abelian Property]\label{lemma:Abelian}
  Let $(\eta, \tau) \in \mathcal{H} \times \mathcal{I}$, 
  fix any finite set $K \subset V$.
If  
  $\alpha$ and $\beta$ are both legal sequences for $\eta$
  that are contained in $K$ and stabilize $\eta$ in $K$,  
  then $m_{K, \eta, \tau} = m_{\alpha} = m_{\beta}$, and 
$M_{K, \eta, \tau} = M_{\alpha} = M_{\beta}$.  
  In particular, $\Phi_{\alpha} \eta = \Phi_{\beta} \eta$.
\end{lemma}

\subsection{Counters}
\label{sec:counters}
It will be useful to identify the jump or sleep instructions placed at specific locations of the array.
For this reason we introduce some very useful variables, which depend on the instruction array and on some indices.
Let $\tau \in \mathcal{I}$ be an array of instructions,  fix a vertex $x \in V$ and an integer  $m \in \mathbb{N}$. 
We let $J^{x,m}_{\tau}$ be the \textit{m-th jump instruction of $\tau$ at $x$}
and $t^{x,m}_{\tau}$ be its corresponding index.
More precisely, we set $t_{\tau}^{x, 0} : = -1$,  and, for any $m \in \mathbb{N}$, we define
\begin{equation}\label{eq:countingjump1}
t^{ x, m}_{\tau}: = \min \{n > t_{\tau}^{x, m-1} \, \, : \, \,  \tau^{x,n} \not= \tau_{x,\rho} \}
\quad \quad \quad 
J^{x,m}_{\tau}:= \tau^{x,t^{x,m}_{\tau}}.  \quad \quad \quad 
\end{equation}
Moreover, for any $m \in \mathbb{N}$ we let $S^{x,m}_{\tau}$ be \textit{the number of sleep instructions of $\tau$ at $x$  between the 
 $m-1$ th  and the $m$ th jump instruction,}
\begin{equation}\label{eq:countingjump2}
S^{x,m}_{\tau} : = t^{x, m}_\tau -    t^{x, m-1}_\tau - 1.
\end{equation}
For example, the variable $S^{x,1}_{\tau}$ represents the number of sleep instructions which are located before the first jump instruction at $x$,
and this variable equals zero if the first instruction at $x$ is a jump instruction.
See Figure \ref{Fig:array} for an example.
\begin{figure}
	\centering
  \includegraphics[scale=0.70]{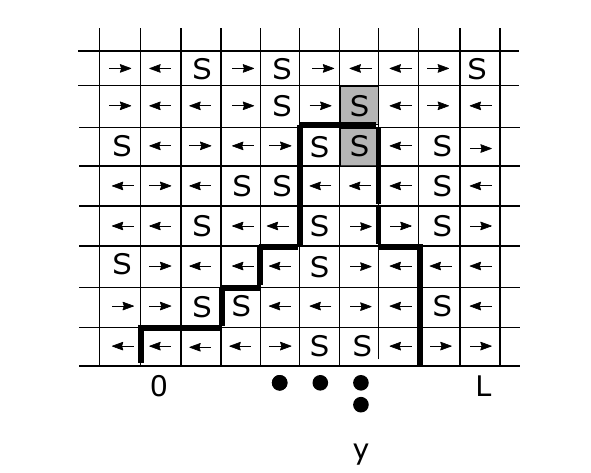}
	\caption{An array  $\tau \in \mathcal{I}$ with $S_\tau^{y,1} = 1$,  $S_\tau^{y,2} = 0$,
$t_\tau^{y,1} = 1$ and $t_\tau^{y,3} = 3$.  
In the figure we assume that the instructions below the bold profile are those which have been used for the stabilisation of $\eta \in \mathcal{H}_a$ in $[0,L]$,
where the particles of $\eta$ correspond to the black circles. 
The array of instructions $\Gamma_-^{y, n}(\tau)$, with $n=M_{K,\eta, \tau}(y)+1 = 5$, is obtained from $\tau$ by `removing' the two dark sleep instructions above the vertex $y$.} 
	  \label{Fig:array}
\end{figure}

\subsection{{Partial orders and monotonicity properties}}
We now introduce a partial order between  particle configurations and arrays of instructions.
Given two particle configurations $\eta, \eta^{\prime} \in \mathcal{H}$, we write $\eta'   \geq   \eta$ if $\eta' (x) \, \geq \, \eta(x)$
for all $x \in V$.
Given two arrays $\tau, \tau^{\prime}$, we write
$
\tau^{\prime} \geq \tau
$
if
$$
\forall x \in V, \quad  \forall m \in \mathbb{N}, \quad J_{\tau}^{x,m} = J_{\tau^{\prime}}^{x,m} \quad 
S_{\tau}^{x,m} \geq  S_{\tau^{\prime}}^{x,m}.
$$
In other words, either $\tau^{\prime} = \tau$  or 
$\tau^{\prime}$ is obtained from $\tau$ by removing some sleep instructions.
The next lemma presents the monotonicity properties of the Diaconis-Fulton representation,  which is a straightforward adaptation of \cite[Lemmas 3 and 5] {Rolla}.

\begin{lemma}[{Monotonicity}]\label{lemma:monotonicity}
   If $K_1 \subset K_2 \subset V$, $\eta \leq \eta'$, $\tau \leq \tau^{\prime}$, then 
   $M_{K_1, \eta, \tau} \leq M_{K_2, \eta', \tau^{\prime}}$
\end{lemma}
By monotonicity, given any growing sequence of subsets $V_1\subseteq V_2 \subseteq V_3\subseteq \cdots \subseteq V$ such that $\lim_{m\to\infty} V_m=V$, 
the limits 
$$
   m_{\eta, \tau}  : = \lim\limits_{m\to \infty} m_{V_m, \eta, \tau}, \quad 
      M_{\eta, \tau}  : = \lim\limits_{m \to \infty} M_{V_m, \eta, \tau},
$$ 
exist and do not depend
on the particular sequence $\{V_m\}_m$.

\vspace{0.25cm}

\subsection{{Probability measure and initial particle distribution}}
We now introduce a probability measure on the space of particle configurations and arrays of instructions. 
The distribution of the initial particle configuration is supported in  $\mathcal{H}_a := \mathbb{N}_0^{V} \subset \mathcal{H}$, it is denoted by $\nu$, 
and is assumed to be a product of Poisson distributions with mean $\mu \in (0, \infty)$.
The parameter $\mu$ then corresponds to the \textit{particle density. }
We also introduce a {probability measure on the set of arrays of instructions}, $\mathcal{I}$.
We denote by $\mathcal{P}_{\lambda}$ the probability measure according to which,
for any $x, y \in V$ and  $j \in \mathbb{N}_0$,
$\mathcal{P}_{\lambda} (  \tau^{x,j} = \tau_{x \rho}  ) = \frac{\lambda}{1 + \lambda}$ and 
$\mathcal{P}_{\lambda} (  \tau^{x,j} = \tau_{xy}   ) = \frac{1}{d_x \, (1 + \lambda)} \mathbbm{1}_{\{ y \sim x \}  }$ independently. 
Finally, we denote by $\mathcal{P}_{\lambda, \mu} =\mathcal{P}_{\lambda}\otimes \nu$ the joint law of
$\eta$ and $\tau$.
We use $\mathbb{P}_{\lambda, \mu}$ to denote the probability measure induced by the ARW process starting from a product of Poisson distributions with parameter $\mu$. 
The following lemma relates the dynamics of ARW to the stability property of the representation.
\begin{lemma}[{Zero-one law, activity and fixation}]
   \label{lemma:01law}
   Let $G =(V,E)$ be an undirected unimodular graph. For every  $x\in V$ we have that,
   \begin{equation}\label{eq:activityandfixation}
   \mathbb{P}_{\lambda, \mu}  (\text{ARW fixates} ) = \mathcal{P}_{\lambda, \mu}^{\nu} ( m_{\eta, \tau} (x) < \infty ) = \mathcal{P}_{\lambda, \mu}^{\nu} ( M_{\eta, \tau} (x) < \infty ) \in \{0, 1\}.
   \end{equation}
\end{lemma}
The lemma was proved  in \cite{Rolla}
in case $G = \mathbb{Z}^d$.
The proof remains unchanged if the graph is vertex transitive
and the initial particle distribution is the product of identical 
distributions.  
Indeed the zero-one law,
$
 \mathcal{P}_{\lambda, \mu}( m_{\eta, \tau} (o) < \infty ) \in \{0,1\},
$
follows from the fact that  the event $\{m_{\eta, \tau} (o) < \infty \}$
is a.s. automorphism invariant and, since  the process is determined by i.i.d. variables at the vertices (initial configuration, sleep and jump instructions), it is then ergodic.
The proof of the other claims is the same as in  \cite{Rolla}.

\section{Relevant events and properties of the jump odometer}
\label{sect:relations}
The goal of this section is to introduce 
the definition of increasing and relevant events and discuss some properties of the jump odometer. 
Recall that $\mathcal{H}$ denotes the set of particle configurations and that $\mathcal{I}$ denotes the set of arrays of instructions.
Let $\mathcal{S}$ be the smallest sigma-algebra generated by all the open subsets of $\mathcal{H} \times \mathcal{I}$ with respect to the natural product topology.
\begin{definition}\label{def:increasing}
We say that an event $\mathcal{A} \in \mathcal{S}$ is  \textit{{increasing}} if
\begin{equation*}\label{eq:increasing}
(\eta, \tau) \in \mathcal{A},  \, \, \, \, \tilde \eta \geq \eta,  \, \, \, \, \tilde \tau \geq \tau \implies \, \, (\tilde \eta,  \tilde \tau) \in \mathcal{A}.
\end{equation*}
\end{definition}

\begin{definition}
\label{def:relevant}
An event  $\mathcal{A} \in \mathcal{S}$  is said to be \textit{relevant} if it can be written as,
\begin{equation}\label{eq:relevantevent}
\mathcal{A} = \{\forall x \in K  \,  \,  M_{K, \eta, \tau}(x) \geq H(x)\}.
\end{equation}
for some finite $K \subset V$ and for some function $H \in  \mathbb{N}_0^K$.
We refer to the set $K$ as  \textit{domain} of $\mathcal{A}$.
\end{definition}
For example, the event 
$\{M_{K, \eta, \tau}(o) \geq L  \}$,
for $L \in \mathbb{N}$ and $o \in K \subset V$, is relevant 
and has domain $K$,
since it can  be written as 
$\{M_{K, \eta, \tau}(x) \geq H(x) \, \, \, \forall x \in K  \}$
for the function  $H \in \mathbb{N}_0^K$ which is such that  $H (x) = L \,   \delta_o(x)$ for any $x \in K$. 
 Note that,  by Lemma \ref{lemma:monotonicity},
 any relevant event is increasing.

We now discuss some properties of the jump odometer and of the relevant events. 
For any arbitrary pair $(y,m) \in V \times \mathbb{N}$ 
we introduce  two operators, $\Gamma^{y,m}_-, \Gamma^{y,m}_1 : \mathcal{I} \mapsto \mathcal{I}$ acting on the instruction array as follows. 
For an arbitrary array $\tau \in \mathcal{I}$,  we let $\Gamma^{y,m}_-(\tau) \in \mathcal{I}$ be the new  array of instructions which is obtained from $\tau$ by \textit{removing all the sleep instruction between the $m-1$th and the $m$th jump instruction at $y$.}
More precisely,  $\Gamma^{y,m}_-(\tau)$ is defined as the unique array such that 
for any $x \in V$ and $k \in \mathbb{N}$, 
$$ 
J_{\Gamma^{y,m}_-(\tau)} ^{x,k}   = J_\tau^{x,k} 
\quad
\mbox{ and } 
\quad 
S^{x,k}_{\Gamma^{y,m}_-(\tau)} 
 =  \begin{cases}
0 &\mbox{ if  $x =  y$ }  \mbox{ and {  $k=m$ } }  \\
S^{x,k}_{\tau}  &\mbox{ otherwise},
\end{cases}  
$$
where we recall that the variables $J^{x,m}_\tau$ and  $S^{x,m}_{\tau}$ 
were defined in Section \ref{sec:counters}.
See Figure \ref{Fig:array} for an example.
Moreover, 
we define a new instruction array $\Gamma^{y,m}_1(\tau) \in \mathcal{I}$, which is  is obtained from $\tau$ by \textit{setting to one the number of sleep instructions between the $m-1$th and the $m$th jump instruction at $y$}.
More precisely, 
$\Gamma^{y,m}_1(\tau)$
 is defined as the unique array such that 
for any $x \in V$ and $k \in \mathbb{N}$, 
$$
J_{\Gamma^{y,m}_1(\tau)} ^{x,k}  = J_\tau^{x,k} 
\quad
\mbox{ and } 
\quad 
S^{x,k}_{\Gamma^{y,m}_1(\tau)}   =  \begin{cases}
1 &\mbox{ if  $x =  y$ }  \mbox{ and {  $k=m$ } }  \\
S^{x,k}_{\tau}  &\mbox{ otherwise}.
\end{cases}  
$$

We now discuss some properties of the jump odometer. 
To begin,   introduce the set 
$
\mathcal{W}
$
of pairs $(\eta, \tau) \in \mathcal{H} \times \mathcal{I}$
such that, for any finite $K \subset V$ and any 
$x \in K$, $m_{K, \eta, \tau}(x) < \infty$. 
Any $(\eta, \tau) \in \mathcal{W}$ is such that 
the stabilisation of any finite set of sites uses a  finite number of instructions and, clearly,  $\mathcal{P}_{\lambda, \mu}(\mathcal{W}) = 1$.
The first simple lemma states that
only the  removal of the sleep instructions which have been used at last at sites during the stabilisation might affect the jump odometer. 

\begin{lemma}\label{prop:essentiality}
\label{lemma:essentiality}
Consider a pair $(\eta, \tau) \in \mathcal{W}$,
let  $K \subset V$ be a finite set, fix an arbitrary vertex $y \in V$. 
For any $n \in \mathbb{N}$ such that  $n \neq M_{K, \eta, \tau}(y)+1$, we have that $$M_{K, \eta, \tau}  =  M_{K, \eta, \Gamma_-^{y, n}(\tau)}.$$
\end{lemma}
\begin{proof}
Fix $(\eta, \tau) \in \mathcal{W}$, $y \in V$ and  $n \in \mathbb{N}$. 
The claim is obvious if $n > M_{K, \eta, \tau}(y) + 1 $,  if $y \not\in  K$ 
or if $S_\tau^{y, n} = 0$,
hence we can assume that $n \leq  M_{K, \eta, \tau}(y)$,  $y \in K$ 
and that $S_\tau^{y, n} > 0$.
We set for brevity $\tau^\prime = \Gamma_-^{y, n}(\tau)$.
We first perform a stabilisation of $K$ under the constraint that no instruction after the $n-1$th jump instruction at $y$ is used. 
More precisely,  we perform a legal sequence of topplings as follows,  namely 
we use 
 any  instruction at $x \in K$ with $x \neq y$
as long as $x$ is unstable, 
and we use any  instruction  at 
 $y$ 
as long as $y$ is unstable \textit{and} 
no   instruction after the 
$n-1$th jump instruction at $y$ is used. 
We iterate this procedure until no further instruction can be,
unless violating such constraints. 
When we have done, we obtain a particle configuration $\eta^\prime$ which is such that,
\begin{equation}\label{eq:desired200}
\eta^{\prime}(x) 
\begin{cases}
\geq 2 & \mbox{ if $x = y$,}  \\
 \in \{0, \rho\}  & \mbox{ if $x \neq  y$}.
\end{cases} 
\end{equation}
The fact that $\eta^\prime(x) \in \{0, \rho\}$ for any $x \neq y$
follows from the definition of our toppling procedure. 
The fact that $\eta^\prime(y) \geq 2$ will be now proved by contradiction.
Indeed, suppose that  this was not true, namely that either {(a)}
$\eta^{\prime}(y) \in \{0,\rho\}$
or {(b)} $\eta^{\prime}(y) = 1$.
If (a) was true, then  we would have stabilized $\eta$ in $K$ using strictly less than $M_{K, \eta, \tau}(y)$ jump instructions at $y$, contradicting our assumption.
Similarly, if (b) was true, then by using one more instruction at $y$ for $\eta$ (which is  a sleep instruction  by assumption) we would have stabilized $\eta$ in $K$ using  strictly less than $M_{K, \eta, \tau}(y)$ jump instructions, contradicting again our assumptions.
This leads to (\ref{eq:desired200}) as desired.
We now complete the stabilisation of $\eta$ in $K$ starting from the particle configuration $\eta^\prime$ and following an arbitrary stabilisation procedure
and denote by $\alpha$ the sequence of instructions which were used for the stabilisation of the initial particle configuration.
Since $\eta^\prime(y) \geq 2$, the next $S_\tau^{y,n}$  instructions at $y$, which are of type sleep by assumption, do not affect  the particle configuration.  
Hence, they  can be removed  from the sequence $\alpha$ and from  the array $\tau$ without  the jump odometer being affected from this removal. This  leads to the new 
array $\tau^{\prime}$ and to a new legal sequence  
 of instructions of $\tau^{\prime}$, $\alpha^{\prime}$, which stabilises $\eta$ in $K$
 and satisfies $M_{\alpha, \tau   } = M_{\alpha^\prime, \tau^\prime}$.
  By the Abelian property, the proof is concluded.
\end{proof}

The next simple lemma states that the jump odometer does not depend on the precise number of 
sleep instructions between any two consecutive jump instructions, but only on whether such a number is zero or strictly positive. 
\begin{lemma}\label{lemma:numberofsleep}
Consider a pair $(\eta, \tau) \in \mathcal{W}$,
let  $K \subset V$ be a finite set, fix  an arbitrary vertex $y \in V$.
For any integer $n \in \mathbb{N}$ such that $S_\tau^{y,n}>0$ we have,
 $$M_{K, \eta, \tau}  =  M_{K, \eta, \Gamma_1^{y, n}(\tau)}.$$
\end{lemma}
\begin{proof}
The proof of the claim is similar to the one of Lemma \ref{prop:essentiality}.
We assume that $S_\tau^{y,n}>1$, the claim is trivial otherwise. 
If $n > M_{K, \eta, \tau}(y)+1$ or $y \not\in K$,  then the proof is trivial.
If  $n < M_{K, \eta, \tau}(y)+1$, 
then the proof is analogous to the proof of Lemma \ref{prop:essentiality}.
Suppose then that $n = M_{K, \eta, \tau}(y)+1$ (a glance at  Figure \ref{Fig:array} may help).
We first perform a stabilisation of $K$ under the constraint that no  instruction after the $n-1$th jump instruction at $y$ is used,  as defined in the proof of Lemma \ref{prop:essentiality}.
We call $\eta^\prime$ the particle configuration we obtain. 
We claim that,
 \begin{equation}\label{eq:desired300}
\eta^{\prime}(x) 
\begin{cases}
\in \{0, 1\} & \mbox{ if $x = y$,}  \\
 \in \{0, \rho\}  & \mbox{ if $x \neq  y$}.
\end{cases} 
\end{equation}
 The fact that $\eta^\prime(x) \in \{0, \rho\}$ for any $x \neq y$
follows from the definition of our toppling procedure. 
The fact that $\eta^\prime(y) \in \{0,  1\}$ will be now proved by contradiction.
Indeed,  if $\eta^\prime(y) > 1$,  then  in order to conclude the stabilisation of $K$
it would be necessary to use at least one additional jump instruction at $x$,  and this would imply that  more than $M_{K, \eta,  \tau} (y)$ jump instructions are used at $y$ for the stabilisation of $K$, thus contradicting our assumptions. 
Moreover, if $\eta^\prime(y) = \rho$, 
then this would mean that the last instruction which was used at $y$ was a sleep instruction and, since this sleep instruction must be located before the  $n-1$th jump instruction at $y$ (unless violating the definition of our toppling procedure), this implies that 
the  initial particle configuration has been stabilised  using strictly less than  
$M_{K, \eta, \tau}(y)$ jump instruction at $y$,  leading again to a contradiction.  This leads to (\ref{eq:desired300}).
We now complete the stabilisation.
In case  $\eta^\prime(y) = 1$,  the use of the next instruction at $y$, which is sleep by assumption,  stabilises the particle configuration $\eta^\prime$ in $K$. Hence,   the removal of the next $S_{\tau}^{y,n}-1 >0$ sleep instructions at $y$ does not affect the jump odometer, since these instructions are not used for the stabilisation.
Similarly, in case   $\eta^\prime(y) = 0$,  we conclude that  $\eta^\prime$  is  already stable and that  no sleep instruction between the $n-1$th and the $n$th jump instruction at $y$ was used for the stabilisation of $\eta$ in $K$. Hence,  by the Abelian property,  also in this case the removal of the next $S_{\tau}^{y,n}-1$  sleep instructions at $y$  from the array does not affect the jump odometer.
This concludes  the proof.
\end{proof}

The next lemma is an immediate application of the previous one and states that any relevant event does not depend on the precise number of sleep instructions which are located between two consecutive jump instructions,  but only on whether this number is zero or strictly positive. 
The lemma also states the obvious fact that the relevant event (\ref{eq:relevantevent}) does not depend on the instructions outside its domain and on the instructions at sites in the domain which are not used for the stabilisation of the domain.  
\begin{lemma}\label{lemma:propertyessentialevents}
Consider an arbitrary finite set $K \subset V$, let $\mathcal{A}$ be any relevant event with domain $K$. 
 For every 
 $(\eta_1, \tau_1), (\eta_2, \tau_2) \in \mathcal{W}$
satisfying
 $\eta_1(x) = \eta_2(x)$
 and 
 $$
 \forall j \in \{ 1, \ldots M_{K, \eta_1, \tau_1}(x) + 1 \}
 \quad S^{x,j}_{\tau_1}  > 0 \iff S^{x,j}_{\tau_2} > 0     \quad  \mbox{ and } \quad 
 J_{\tau_1}^{x,j} = J^{x,j}_{\tau_2},
 $$
  for every $x \in K$,
we have that 
 \begin{equation}
\label{eq:desiredrelevant}
 M_{K, \eta_2, \tau_2} =  M_{K, \eta_1, \tau_1}.
 \end{equation}
This in turn implies that, 
 \begin{equation}\label{eq:desiredrelevant2}
(\eta_1, \tau_1) \in \mathcal{A} \iff  (\eta_2, \tau_2) \in \mathcal{A}.
 \end{equation}
\end{lemma}
\begin{proof}
Let $(\eta_1, \tau_1)$, $(\eta_2, \tau_2)$ be as in the assumptions of the lemma.
Let $U $ be the set of pairs $(x, n) \in K \times \mathbb{N}$ such that $ 1 \leq n \leq  M_{K, \eta_1, \tau_1}(x)+1 $
and $S_{\tau_1}^{x, n} > 0.$
Let $(x_1, n_1), \ldots (x_k, n_k)$ be a sequence of elements in $U$ such that each element of $U$ appears one time in  the sequence. 
Define the arrays $\tau^\prime_1 :=   \Gamma_{1}^{x_1, n_1   }   \Gamma_{1}^{x_2, n_2   } \ldots  \Gamma_{1}^{x_k, n_k} \tau_1$
and $\tau^\prime_2 :=   \Gamma_{1}^{x_1, n_1   }   \Gamma_{1}^{x_2, n_2   } \ldots  \Gamma_{1}^{x_k, n_k} \tau_2$.
By
Lemma \ref{lemma:numberofsleep}
we deduce that
$M_{K, \eta_1, \tau_1} = M_{K, \eta_1, \tau_1^\prime}$
and that 
$M_{K, \eta_2, \tau_2} = M_{K, \eta_2, \tau_2^\prime}$.
Moreover, note that if $(y,j)$ belongs to a sequence  
of instructions of $\tau^\prime_1$ which stabilises $\eta_1$ in $K$, 
by our assumptions on $\tau_1$, $\tau_2$
and by the fact that  $j \leq M_{K, \eta_1, \tau_1}(y)+1$
we then have that 
${\tau^\prime_1}^{y,j} = {\tau^\prime_2}^{y,j}$.
Hence, by the Abelian property and by the fact that $\eta_1$ and $\eta_2$ are identical in $K$ we have that 
$M_{K, \eta_1, \tau_1^\prime} = 
M_{K, \eta_2, \tau^\prime_2}$. 
Combining the identities we derived so far we obtain   (\ref{eq:desiredrelevant}),  as desired.   
The claim
 (\ref{eq:desiredrelevant2}) now  follows immediately from the definition of relevant event.
\end{proof}

\section{Partial derivatives}
\label{sect:Russo}
In this section we introduce the notion of particle and sleeping essential pairs 
and present our formula for the partial derivatives of the probability of  relevant events. 
For an arbitrary particle configuration $\eta \in \mathcal{H}$,
vertex
 $x \in V$ and integer $k \in \mathbb{N}_0$, we denote by
$\eta^{(x,k)} \in \mathcal{H}$ the particle configuration which is obtained from $\eta$ by setting $k$ (active) particles at $x$,  i.e, 
$$
\eta^{(x,k)}(y) := 
\begin{cases}
k  &\mbox{ if } y = x, \\
\eta(y) & \mbox{ if } y \neq x.
\end{cases}
$$
Let now $\mathcal{A} \in \mathcal{S}$ be an arbitrary event.

\begin{definition}
For every vertex $x \in V$ and integer $k \in \mathbb{N}_0$, 
we define the event $\{$the pair $(x,k)$ is \textit{particle-essential} for the event $\mathcal{A} \}$ as the set of realisations  $(\eta, \tau) \in \mathcal{H} \times \mathcal{I}$ such that,
$$
( \eta^{(x,k)}, \tau) \not\in \mathcal{A} \, \, \mbox{ and } \, \,  ( \eta^{(x,k+1)}, \tau) \in \mathcal{A}.
$$
\end{definition}
Sometimes, we will write p-essential in place of particle-essential.
\begin{definition}
For every vertex $x \in V$ and integer $k \in \mathbb{N}$, 
we define the event 
$\{$the pair $(x,  k)$ is \textit{sleeping-essential for $\mathcal{A}$} $\}$
 as the set of realisations $(\eta, \tau) \in \mathcal{H} \times \mathcal{I}$ such that 
$$
( \eta, \Gamma_1^{x,k}(\tau)) \notin  \mathcal{A} \, \, \mbox{ and } \, \,  \big  ( \eta, \Gamma^{x,k}_-(\tau) \big  ) \, \in \mathcal{A}.
$$
\end{definition}
Sometimes, we will write s-essential in place of sleeping-essential.

\begin{remark}\label{remark:dependences}
It follows from the definition of particle-essential  and  sleeping-essential pairs that, given an arbitrary event $\mathcal{A} \in \mathcal{S}$, 
the event $\{ (y,k)$ is a particle-essential pair  for $\mathcal{A}\}$ is independent from   
the initial number of particles at $y$,$\eta(y)$, and that the event
$\{(y,k)$ is a sleeping-essential pair  for $\mathcal{A}\}$ is independent from the variable $S_{\tau}^{y,k}$.
\end{remark}

We now present the main theorem of this section.
 In the statement of the theorem, we consider the probability of an event $\mathcal{A} \in \mathcal{S}$, $\mathcal{P}_{\lambda, \mu}( \mathcal{A})$, as a function of $(\lambda, \mu)$ in  $\mathbb{R}^2_0$.
For every $j \in \mathbb{N}_0$, we  denote by $\nu_j= \nu_j(\mu) = e^{- \mu}  \mu^j / j!$ the probability that a single vertex hosts $j$ particles,  we define  $\nu_{ > j} = \nu_{ > j}(\mu) : = \sum_{\ell > j}  \nu_\ell$, the probability that a vertex hosts more than $j$ particles and let $\nu^\prime _{>j} := \frac{d}{d \mu}  \nu_{>j}(\mu)$ be its  derivative with respect to $\mu$, noting that $\nu^\prime _{>j} = \nu_j$.
\begin{theorem}\label{theo:Russo}
Let $\mathcal{A}$ be any relevant event. 
The  function $\mathcal{P}_{\lambda, \mu}( \mathcal{A})$ is differentiable in $\mathbb{R}_+^2$ and,  for every 
 $( \lambda^\prime, \mu^\prime )  \in \mathbb{R}_+^2$, we have that,
\begin{align}
\label{eq:partial1}
 \frac{\partial}{\partial \lambda } \, \mathcal{P}_{\lambda, \mu} \big (  \mathcal{A} \big )
\Big \vert_{\lambda^{\prime}, \mu^{\prime}}  = &  - 	(\frac{1}{1+\lambda'})^2 \,   \sum\limits_{y \in V, j \in \mathbb{N}} \mathcal{P}_{\lambda', \mu^{\prime}} \big (  (y,j) \, \mbox{ is sleeping-essential for } \mathcal{A} \big ),   \\
\label{eq:partial2}
 \frac{\partial}{\partial \mu }\, \mathcal{P}_{\lambda, \mu} \big (  \mathcal{A} \big )\Big \vert_{\lambda^{\prime}, \mu'}  = &      \sum\limits_{y \in V, j \in \mathbb{N}_0}  \mathcal{P}_{\lambda^{\prime}, \mu'} \big (  (y,j)\, \mbox{ is particle-essential for } \mathcal{A} \big )  \, \,   \nu'_{ > j}(\mu^{\prime}).
\end{align}
\end{theorem}
The remainder of this section is devoted to  the proof of Theorem \ref{theo:Russo}, which is divided into  three subsections. In Section \ref{sect:coupling} we  introduce a coupling which allows the comparison of ARW-systems with different values of the parameters $\mu$ and $\lambda$. In Sections \ref{sect:proof1}
and \ref{sect:proof2} we will use such a coupling to present the proof of (\ref{eq:partial1}) and (\ref{eq:partial2}) respectively.
From now on, we will write $\frac{\partial}{\partial \lambda } \, \mathcal{P}_{\lambda, \mu} \big (  \mathcal{A} \big )$ in place of $\frac{\partial}{\partial \lambda^\prime } \, \mathcal{P}_{\lambda^\prime, \mu^\prime} \big (  \mathcal{A} \big )  \big \vert_{\lambda, \mu}$, sometimes we will write $\partial_\lambda$ for $\frac{\partial}{\partial \lambda }$, and we will do the same for the partial derivative with respect to $\mu$.

\subsection{Probability space for coupled activated random walk models}
\label{sect:coupling}
We now introduce a new probability space which allows us to couple activated random walk
systems corresponding to different values of $\mu  \geq 0 $ and $\lambda  \geq 0$.
This new probability space will be denoted by $(\Sigma, \mathcal{F}, \boldsymbol{\mathcal{P}})$.
To begin, let $ \big ( X_{x} \big ) _{x \in V}$,  $ \big ( Y_{x, m} \big ) _{x \in V, m \in \mathbb{N}}$,
and $\big ( A_{x,m} \big )_{x \in V, m \in \mathbb{N}}$ be three sequences of independent random variables in $(\Sigma, \mathcal{F}, \boldsymbol{\mathcal{P}})$
which are distributed as follows.
The variables  $ \big ( X_{x} \big ) _{x \in V}$  and $ \big ( Y_{x, m} \big ) _{x \in V, m \in \mathbb{N}}$
have uniform distribution in $[0,1]$, while the variables 
 $\big ( A_{x,m} \big )_{x \in V, m \in \mathbb{N}}$  are such that, for each $x \in V$,  and $m \in \mathbb{N}$, $A_{x,m}$ takes values in $\{   \tau_{xy} \, : \, x \in V, y \sim x \}$,
and has distribution
$$
\boldsymbol{\mathcal{P}} ( A_{x,m}  =   \tau_{x y} ) = \frac{1}{d_x}.
$$ 
The variables $(X_x)_{x \in V}$ will be used to sample the initial particle configurations, the variables $(Y_{x,m})_{x \in V, m \in \mathbb{N}}$  will be used to sample the sleep instructions, and the variables $(A_{x,m})_{x \in V, m \in \mathbb{N}}$  will correspond to the jump instructions.
We start with the construction of the  \textit{{initial particle configuration.}}
 Define the function $\eta_{\mu} : \Sigma \rightarrow  \mathcal{H}$, which depends
 on the parameter  $\mu  \geq 0$.
For every $x \in V$, let $k \in \mathbb{N}_0$ be the unique integer such that $X_x \in  [ \nu_{< k}(\mu), \nu_{< k+1}(\mu) )$,
where $\nu_{<k} (\mu):= 1 - \nu_{> k-1} (\mu)$ and $\nu_{<0}(\mu) := 0$. 
Then, $\eta_{\mu}(x) := k$.
Note that,  it follows by construction that,
\begin{equation}\label{eq:coupling number part}
\forall \mu \geq 0, \quad \forall x \in V, \quad   \boldsymbol{\mathcal{P}} \big (  \eta_{\mu}(x) = k \big )  =  \nu_{k}(\mu),
\end{equation}
and that the variables $\big ( \eta_{\mu}(x) \big )_{x \in V}$
 are independent.
We now construct the \textit{array of instructions.}
For every $m \in \mathbb{N}$ and $x \in V$, we define the functions
$R_{\lambda}^{x, m} : \Sigma \rightarrow \mathbb{N}$,
which represent the number of sleep instructions between the $m-1$-th and the $m$th jump instruction at $x$ and 
depend on the parameter $\lambda \in [0, \infty)$,
$$R_{\lambda}^{x, m} : =
\begin{cases}
\ell\, \, \mbox{ if} \, \, \, Y_{x,m} \in \big ( \, (\frac{\lambda}{1+\lambda})^{\ell+1},  (\frac{\lambda}{1+\lambda})^{\ell} \big  ] \\
0\, \, \mbox{ otherwise. }
\end{cases}
$$
Note that, by construction, 
\begin{equation}\label{eq:coupling number sleep}
\forall \lambda \in \mathbb{R}^+_{0}, \quad \forall x \in V, \quad \forall m \in \mathbb{N}, \quad \forall \ell \in \mathbb{N}_0, \quad  \boldsymbol{\mathcal{P}} \big (  R_{\lambda}^{x, m}  = \ell  \big )  = \frac{1}{1+\lambda} \, \, \big (  \frac{\lambda}{1 + \lambda} \big )^{\ell},
\end{equation}
and that the variables $\big ( R_{\lambda}^{x, m}\big ) _{x \in V, m \in \mathbb{N}}$ are independent.
Moreover, we define the function
$\tau_{\lambda} : \Sigma \rightarrow \mathcal{I}$,
which represents the instruction array for the coupled activated random walk systems as the unique array of instructions such that  
$$
\forall x \in V, \quad 
\forall m \in \mathbb{N,} \quad 
J^{x,m}_{\tau_{\lambda}} : = A_{x, m} \quad 
\mbox{ and }
\quad S^{x,m}_{\tau_{\lambda}}  : = R^{x,m}_{\lambda}.
$$
By construction, we proved the following proposition.
\begin{proposition}\label{prop: coupling}
Let $\lambda \in [0, \infty)$ and $\mu \in [0, \infty)$.  Sample the pair $(\eta, \tau) \in \mathcal{H} \times \mathcal{I}$ according to $\mathcal{P}_{\lambda, \mu}^{\nu}$ and let
$\eta_{\lambda} : \Sigma \rightarrow \mathcal{H}$ and 
$\tau_{\lambda} : \Sigma \rightarrow \mathcal{I}$ 
be the random variables in the probability space $(\Sigma, \mathcal{F}, \boldsymbol{\mathcal{P}})$ which  have been defined above. We have that,
$$
(\eta, \tau)  
\overset{d}{=} (\eta_{\mu}, \tau_\lambda),
$$
where `$ \, \overset{d}{=}$' denotes equality in distribution.
From this, we deduce that, for every event $\mathcal{A} \in \mathcal{S}$, 
$
\boldsymbol{\mathcal{P}} \big ( (  \eta_\mu,  \tau_\lambda     ) \in  \mathcal{A}   \big ) =  \mathcal{P}_{ \mu, \lambda} \big ( (  \eta,  \tau     ) \in  \mathcal{A}   \big ).
$
\end{proposition}
In the next two subsections we will use this coupling to prove equations (\ref{eq:partial1}) and (\ref{eq:partial2}).

\subsection{Proof of equation (\ref{eq:partial1})}
\label{sect:proof1}
Let $\mathcal{A}$ be any relevant event with (finite) domain $K \subset V$. 
To begin, we deduce from Proposition \ref{prop: coupling}
that,  since $\mathcal{A}$ is increasing,  then for any $\delta >0$,
\begin{align}
\begin{split}
\label{eq:first term}
\mathcal{P}_{\lambda, \mu}    \big  (\mathcal{A}  \big  ) - 
\mathcal{P}_{\lambda + \delta, \mu} \big ( \mathcal{A}  \big  )  
=
\sum\limits_{  \substack{ \tilde M  \in \mathbb{N}_0^K  }}  \boldsymbol{\mathcal{P}} \big (   
M_{K, \eta_\mu, \tau_0} = \tilde M, 
(\eta_\mu, \tau_\lambda) \in \mathcal{A}, 
(\eta_\mu, \tau_{\lambda + \delta}) \notin \mathcal{A} \big ),
\end{split}
\end{align}
where,  since $M_{K, \eta, \tau}(y) = 0$ for $y \not\in K$ a.s,
by a slight abuse of notation we consider the function $M_{K, \eta, \tau}$ as taking values in 
$\mathbb{N}_0^K$ rather than in  $\mathbb{N}_0^V$. 
 Recall the coupling construction which was defined in Section \ref{sect:coupling}.
For arbitrary $(y,j) \in K \times \mathbb{N}$,
 we define the events
in the probability space $\boldsymbol{\mathcal{P}}$,
\begin{equation}
\label{eq:eventsB}
\begin{split}
&  \mathcal{B}^{y,j, +}   := \{  R_{\lambda + \delta}^{y,j} > 0 \} \cap \{ 
 R_{\lambda}^{y,j} = 0 \},
 \\
&  \mathcal{B}^{y,j, -}_{\tilde M}   :=  \bigcap_{   \substack{ (x, k)  \in K \times \mathbb{N} : \\ 
k\leq \tilde M(x) + 1   \\ 
(x, k ) \neq (y,j)}   }   \Big \{   R_{\lambda + \delta}^{x,k} >0 \iff R_{\lambda}^{x,k}>0 \Big \}, \\
&  \mathcal{B}^{2}_{\tilde M}   := \bigcup_{ \substack{(v, n), (x,k) \in K \times \mathbb{N} : \\ 
 (v,n) \neq (x,k), \,
n  \leq   \tilde M(v) + 1, \,
k \leq  \tilde M(x) +1 
} } \Big \{ R_{\lambda + \delta}^{v,n} > 0, \, 
 R_{\lambda + \delta}^{x,k} > 0, \,
 R_{\lambda}^{v,n} = 0,   \, 
 R_{\lambda}^{x,k} = 0 \,  \Big \}.
\end{split}
\end{equation}

Using  the  fact that  $\mathcal{A}$ is 
relevant, applying Lemma \ref{lemma:propertyessentialevents} and the conditional probability formula, we obtain that 
\begin{multline}
 \boldsymbol{\mathcal{P}} \big (   
 M_{K, \eta_\mu, \tau_0 } = \tilde M, 
 (\eta_\mu, \tau_\lambda) \in \mathcal{A}, 
(\eta_\mu, \tau_{\lambda + \delta}) \notin \mathcal{A}
\,  \big ) = \\
\label{eq:two terms}
 \sum\limits_{ \substack{ ( y,j) \in K \times \mathbb{N} : \\  j \leq \tilde M(y)+1 }}   \boldsymbol{\mathcal{P}} \big (   \,  
M_{K, \eta_\mu, \tau_0 } = \tilde M,  (\eta_\mu, \tau_\lambda) \in \mathcal{A}, 
(\eta_\mu, \tau_{\lambda + \delta}) \notin \mathcal{A},  \mathcal{B}_{\tilde M}^{y,j,-}   \Big | \, \, 
\mathcal{B}^{y,j,+} \, \big )  
\boldsymbol{\mathcal{P}} (   \mathcal{B}^{y,j,+}\,  )
\\  + 
  \boldsymbol{\mathcal{P}} \big (   
M_{K, \eta_\mu, \tau_0 } = \tilde  M,  (\eta_\mu, \tau_\lambda) \in \mathcal{A}, 
(\eta_\mu, \tau_{\lambda + \delta}) \notin \mathcal{A},   \,  \mathcal{B}_{\tilde M}^{2} \, \big ).
\end{multline}
We now let  $f_\delta (\tilde M)$ and 
 $u_\delta (\tilde M)$
 be the respectively the first and second term in the  right-hand side (RHS) of the previous expression.
Thus we deduce from (\ref{eq:first term}) and (\ref{eq:two terms}) that,
\begin{equation}\label{eq:xterms}
\lim_{\delta \rightarrow 0^+} \frac{1}{\delta}  \big [
\mathcal{P}_{\lambda, \mu}    \big  (\mathcal{A}  \big  ) - 
\mathcal{P}_{\lambda + \delta, \mu} \big ( \mathcal{A}  \big  )  \big ] 
=
\lim_{\delta \rightarrow 0^+} 
\sum\limits_{ \tilde M \in \mathbb{N}_0^K   }  
 \frac{1}{\delta}
 f_\delta (\tilde M)
 +   \, 
\lim_{\delta \rightarrow 0^+} 
\sum\limits_{ \tilde M \in \mathbb{N}_0^K   }  
 \frac{1}{\delta}
u_\delta (\tilde M).
\end{equation}
We now consider the two terms in the  right-hand side (RHS) of the previous identity separately. 

\subsubsection{First term in the RHS of (\ref{eq:xterms})}
To begin, note that 
 in the limit as $\delta \rightarrow 0^+$,  uniformly in $y \in V$ and $j \in \mathbb{N}$,
\begin{equation}\label{eq:orderdelta}
\boldsymbol{\mathcal{P}} (   \mathcal{B}^{y,j,+}  ) = 
\boldsymbol{\mathcal{P}} \Big (   Y_{y,j} \in \big ( \, \frac{\lambda}{1+\lambda}, 1 \big ] \setminus ( \, \frac{\lambda + \delta}{1 + \lambda + \delta},1  \, \big ]  \Big ) = 
\delta \, \big ( \frac{1}{1+\lambda} \big )^2  + O(\delta^2).
\end{equation}
Consider now an arbitrary $\tilde M \in \mathbb{N}_0^K$. 
Using independence, the definition of sleeping-essential pair,  the important Remark  \ref{remark:dependences} for the second identity,  
and the fact that for each given $\tilde M \in \mathbb{N}_0^K$, 
$\lim_{\delta \rightarrow 0^+} \boldsymbol{\mathcal{P}}( \mathcal{B}_{\tilde M}^{y,j,-} ) = 1$ for the third identity,  we obtain that,
\begin{multline}\label{eq:essential limit} 
 \lim\limits_{\delta \rightarrow 0^+}  \frac{1}{\delta}  f_{\delta} ( \tilde M)   \\  
 =   \lim\limits_{\delta \rightarrow 0^+}  \boldsymbol{\mathcal{P}} \big (   \, 
M_{K, \eta_\mu, \tau_0 } = \tilde M,  (\eta_\mu, \tau_\lambda) \in \mathcal{A}, 
(\eta_\mu, \tau_{\lambda + \delta}) \notin \mathcal{A},  \, \,  \mathcal{B}_{\tilde M}^{y,j,-}   \Big | \, \, 
\mathcal{B}^{y,j,+} \, \big ) \big ( \frac{1}{1+\lambda} \big )^2        \\ 
=
 \lim\limits_{\delta \rightarrow 0^+}  \boldsymbol{\mathcal{P}} \big (   M_{K, \eta_\mu, \tau_0} = \tilde M, 
 (\eta_\mu, \Gamma^{y,j}_{-}( \tau_{\lambda})) \in \mathcal{A},  
(\eta_\mu,  \Gamma^{y,j}_1( \tau_{\lambda + \delta})) \notin \mathcal{A} , \,  \mathcal{B}_{\tilde M}^{y,j,-}  \ \,  )  \big ( \frac{1}{1+\lambda} \big )^2    \\ = 
  \mathcal{P}_{\lambda, \mu} \big (  \{  (y,j) \mbox { is s-essential for } \mathcal{A}  \} \cap \{   M_{K, \eta_\mu, \tau_0} = \tilde M\}   \big) \big ( \frac{1}{1+\lambda} \big )^2.
  \end{multline}
Moreover,  from (\ref{eq:orderdelta}) we deduce that for 
 any $\tilde M \in \mathbb{N}^K_0$,
$$
\frac{1}{\delta} f _{\delta} ( \tilde M) \leq  
 \sum\limits_{ y \in K}   ( \tilde M(y) + 1)
 \boldsymbol{\mathcal{P}} \big (   
M_{K, \eta_\mu, \tau_0 } = \tilde M  \, \big ) \, 
\big ( \frac{1}{(1+\lambda)^2}  + O(\delta) \big ).
$$
Since the quantity  in the RHS of the previous expression
is summable in $\tilde M$ and the sum is uniformly bounded in $\delta \in (0, 1)$, 
we can use  (\ref{eq:essential limit}) and apply the  dominated convergence theorem to conclude that,
\begin{multline*}
  \lim\limits_{\delta \rightarrow 0^+} 
  \sum\limits_{ \tilde M \in \mathbb{N}_0^K   } 
  \frac{1}{\delta} f _{\delta} ( \tilde M)
  = 
  \sum\limits_{ \tilde M \in \mathbb{N}_0^K   } 
      \lim\limits_{\delta \rightarrow 0^+} 
  \frac{1}{\delta} f _{\delta} ( \tilde M) \\
  = 
  (\frac{1}{1 + \lambda})^2 \, 
\sum\limits_{ \tilde M \in \mathbb{N}_0^K   } 
\sum\limits_{ \substack{ ( y,j) \in K \times \mathbb{N} : \\  j \leq \tilde M(y)+1 }}  
 \mathcal{P}_{\lambda, \mu} \big (  \{  (y,j) \mbox { is s-essential for } \mathcal{A}  \} \cap \{   M_{K, \eta_\mu, \tau_0} = \tilde M\}   \big) \\  
 =  (\frac{1}{1 + \lambda})^2 \, 
\sum\limits_{ \substack{ ( y,j) \in K \times \mathbb{N} }}  
 \mathcal{P}_{\lambda, \mu} \big (   (y,j) \mbox { is s-essential for } \mathcal{A}   \big).
  \end{multline*}
To conclude the proof it remains then to show that the second term in the RHS of 
(\ref{eq:xterms}) equals zero. 
  
  \subsubsection{Second term in the RHS of (\ref{eq:xterms})}
  We now prove that the second term in the RHS of 
(\ref{eq:xterms}) equals zero. 
For this, note that in the limit as $\delta \rightarrow 0^+$,
\begin{multline}
\label{eq:secondordertermlambda}
\sum\limits_{  \substack{ \tilde M \in  \mathbb{N}_0^K}}  u_{\delta}( \tilde M) \leq \sum\limits_{  \substack{ \tilde M \in  \mathbb{N}_0^K }} 
\boldsymbol{\mathcal{P}} \big (   M_{K, \eta_\mu, \tau_0} = \tilde M,  \mathcal{B}_{\tilde M}^{2} \, 
 \big ) \\
 \leq 
\sum\limits_{  \substack{ \tilde M \in  \mathbb{N}_0^K}}   
\, \,   \sum\limits_{ \substack{ (x, k), (y,j) \in K \times \mathbb{N} : \\ (x,k) \neq (y,j) \\ k \leq 
\tilde M(x) + 1, j \leq \tilde M(y) + 1}}
\boldsymbol{\mathcal{P}} \Big (   Y_{y,j},  Y_{x,k}  \in \big ( \, \frac{\lambda}{1+\lambda}, 1 \big ] \setminus ( \, \frac{\lambda + \delta}{1 + \lambda + \delta},1  \, \big ],     M_{K, \eta_\mu, \tau_0} = \tilde M
 \Big ),
  \\
= 
\big (  \delta^2  ( \frac{1}{1 + \lambda})^4 + o (\delta^2) \big )
\sum\limits_{ \substack{ (x, k), (y,j) \in K \times \mathbb{N}_0  : \\ (x,k) \neq (y,j) }} 
\boldsymbol{\mathcal{P}} \Big (   M_{K, \eta_\mu, \tau_0} (x)  \geq k, 
 M_{K, \eta_\mu, \tau_0} (y)  \geq j   \Big ),
\end{multline}
where
for the first inequality we used the union bound, 
for the first identity we used the independence
between the function $M_{K, \eta_\mu, \tau_0}$ 
and the functions $Y_{y,j}$.
Now note that, since $K$ is finite, 
then the sum in the last expression is finite and depends only on  $\mu$ and $K$.
This implies that the second term in the RHS of (\ref{eq:xterms}) equals zero and concludes the proof of the right partial derivative. 
 
 \subsubsection{The left partial partial derivative}
To see that (\ref{eq:partial1}) holds  with  `$\delta \rightarrow 0^-$'
in place of `$\delta \rightarrow 0$' we observe that,  since the event $\mathcal{A}$ is increasing, 
then  for any  $\delta > 0$ we have that,
\begin{align*}
\mathcal{P}_{\lambda - \delta, \mu} \big ( \mathcal{A}  \big  ) - \mathcal{P}_{\lambda, \mu}    \big  (\mathcal{A}  \big  )   = 
\sum\limits_{\tilde M \in \mathbb{N}_0^K} 
\boldsymbol{\mathcal{P}} \big (   \, M_{K, \eta_\mu \tau_0} = \tilde M,
(\eta_\mu, \tau_\lambda) \not\in \mathcal{A}, 
(\eta_\mu, \tau_{\lambda - \delta}) \in \mathcal{A}  \, \big ).
\end{align*}
Now the proof follows the same steps as the proof of  the right partial derivative,
with the following events
\begin{align*}
&  \mathcal{C}^{y,j, +}   := \{  R_{\lambda - \delta}^{y,j} = 0 \} \cap \{ 
 R_{\lambda}^{y,j} > 0 \},
 \\
&  \mathcal{C}^{y,j, -}_{\tilde M}   :=  \bigcap_{   \substack{ (x, k)  \in K \times \mathbb{N} : \\ 
k\leq \tilde M(x) + 1   \\ 
(x, k ) \neq (y,j)}   }   \Big \{   R_{\lambda - \delta}^{x,k} >0 \iff R_{\lambda}^{x,k}>0 \Big \}, \\
&  \mathcal{C}^{2}_{\tilde M}   := \bigcup_{ \substack{(v,n), (x,k) \in K \times \mathbb{N} : \\ 
 (v,n) \neq (x,k),
n  \leq   \tilde M(v) + 1,
k \leq  \tilde M(x) +1
} } \Big \{ R_{\lambda - \delta}^{v,n} = 0, \, 
 R_{\lambda - \delta}^{x,k} = 0, \,
 R_{\lambda}^{v,n} > 0,   \, 
 R_{\lambda}^{x,k} > 0 \,  \Big \}.
\end{align*}
playing the role of those defined in  (\ref{eq:eventsB}) 
for every $(y,j) \in K \times \mathbb{N}$ and $\tilde M \in \mathbb{N}_0^K$, 
and 
\begin{multline*}
 \lim\limits_{\delta \rightarrow 0^+}  
  \boldsymbol{\mathcal{P}} \big (   \,  
M_{K, \eta_\mu, \tau_0 } = \tilde M,  (\eta_\mu, \tau_{\lambda - \delta} ) \in \mathcal{A}, 
(\eta_\mu, \tau_{\lambda }) \not\in \mathcal{A}, \,   \mathcal{C}_{\tilde M}^{y,j,-}   \Big | \, \, 
\mathcal{C}^{y,j,+} \, \big )    =   \\ 
 \lim\limits_{\delta \rightarrow 0^+}  \boldsymbol{\mathcal{P}} \big (   \,  M_{K, \eta_\mu, \tau_0} = \tilde M, \,   (\eta_\mu, \Gamma^{y,j}_{-}( \tau_{\lambda - \delta})) \in \mathcal{A},  
(\eta_\mu,  \Gamma^{y,j}_{1}( \tau_{\lambda})) \not\in \mathcal{A},   \mathcal{C}_{\tilde M}^{y,j,-}  \,  )  = \\
  \mathcal{P}_{\lambda, \mu} \big (  \{  (y,j) \mbox { is s-essential for } \mathcal{A}  \} \cap \{   M_{K, \eta_\mu, \tau_0} = \tilde M\}   \big).
\end{multline*}
playing the role of (\ref{eq:essential limit}).
Since we obtain the same formula for the right and left partial derivative, the proof is concluded.

\subsection{Proof of equation (\ref{eq:partial2})}
\label{sect:proof2}
We now turn to the proof of the second partial derivative. 
Let  $\mathcal{A}$ be an arbitrary relevant event. 
From Proposition \ref{prop: coupling}, by the fact that $\mathcal{A}$ is increasing
we obtain that
for any $\delta >0$,
\begin{equation}\label{eq:1}
\mathcal{P}_{\lambda, \mu+\delta}(\mathcal{A}) - 
\mathcal{P}_{\lambda, \mu}(\mathcal{A}) = \sum\limits_{\tilde \eta \in \mathcal{H}_a} \boldsymbol{\mathcal{P}} \big (  (\eta_{\mu + \delta}, \tau_\lambda) \in \mathcal{A},  
(\eta_{\mu }, \tau_\lambda) \not\in \mathcal{A}, \eta_\mu = \tilde \eta \big ).
\end{equation}
To begin, we define the sets,
\begin{equation}\label{eq:setsE}
\begin{split}
 & \mathcal{E}^{x,+}   := \{ \eta_{\mu+\delta}(x) > \eta_\mu(x)   \} , \\
 & \mathcal{E}^{x,-}  : = \{\forall y \in K \setminus \{x\},  
\eta_{\mu+\delta}(y) = \eta_\mu(y)    \}, \\
& \mathcal{E}^{2}   : = \{ \exists x_1, x_2 \in K \, \, : \, \, 
x_1 \neq x_2 \mbox{ and } \eta_{\mu + \delta}(x_1) >\eta_\mu(x_1),  \,  \eta_{\mu + \delta}(x_2) > \eta_\mu(x_2) \,  \}, \\
& \mathcal{E}^{x}_{\tilde \eta}   := \{ \eta \in \mathcal{H} \, \ : \,  \forall y \in K \setminus \{x\} \, \, \, 
\eta_\mu(y) = \tilde \eta(y)  \},
\end{split}
\end{equation}
where the first three sets are elements of the sigma-algebra $\mathcal{F}$, the last set is defined as a subset of $\mathcal{H}$, and  $\tilde \eta \in \mathcal{H}$.
Now note that, for each $\tilde \eta \in \mathcal{H}_a$, 
we have that,
\begin{multline}\label{eq:equationgoal}
\boldsymbol{\mathcal{P}} \Big (  \, (\eta_{\mu + \delta}, \,\,  \tau_\lambda) \in \mathcal{A}, \,  \, 
(\eta_{\mu }, \tau_\lambda) \not\in \mathcal{A}, \, \, \eta_\mu = \tilde \eta \, \Big ) =  \\
\sum\limits_{x \in K} \boldsymbol{\mathcal{P}} \Big (  (\eta_{\mu + \delta}, \tau_\lambda) \in \mathcal{A},  \, 
(\eta_{\mu }, \tau_\lambda) \not\in \mathcal{A},\,   \mathcal{E}^{x,-}, \, 
\eta_\mu \in \mathcal{E}^x_{\tilde \eta} \, 
\Big | \, \,\mathcal{E}^{x,+}, \eta_\mu(x) = \tilde \eta(x)  \Big ) \, \, 
\\ \times \boldsymbol{\mathcal{P}} \Big (   \mathcal{E}^{x,+}, \eta_\mu(x) = \tilde \eta(x) \Big ) \, \,  + \, \, 
\boldsymbol{\mathcal{P}} \Big (  (\eta_{\mu + \delta}, \tau_\lambda) \in \mathcal{A},  
(\eta_{\mu }, \tau_\lambda) \not\in \mathcal{A},
  \mathcal{E}^2, \eta_\mu = \tilde \eta \Big ).
\end{multline}
We now let $g_{\delta}(\tilde \eta)$
and  $h_{\delta}(\tilde \eta)$ respectively be the first and second term in the RHS of the previous expression. 
Using (\ref{eq:equationgoal})  in (\ref{eq:1})
we then deduce that 
\begin{equation}\label{eq:expressionmu}
\lim_{\delta \rightarrow 0^+} 
\frac{\mathcal{P}_{\lambda, \mu+\delta}(\mathcal{A}) - 
\mathcal{P}_{\lambda, \mu}(\mathcal{A})}{\delta} 
= 
\lim_{\delta \rightarrow 0^+}  
\sum\limits_{\tilde \eta \in \mathcal{H}_a   }
\frac{1}{\delta} \,  g_{\delta} (\tilde \eta) 
 +
\lim_{\delta \rightarrow 0^+}  
\sum\limits_{\tilde \eta \in \mathcal{H}_a   }
\frac{1}{\delta} \,  h_{\delta} (\tilde \eta).
\end{equation}
We now consider the two terms in the RHS of the previous expression  separately.
\subsubsection{First term in the RHS  of  (\ref{eq:expressionmu})}
To begin,   from Proposition \ref{prop: coupling} and from a simple computation
we deduce that, in the limit as 
$\delta \rightarrow 0^+$,
\begin{align}\label{eq:equation2}
\begin{split}
\boldsymbol{\mathcal{P}} \Big (\eta_{\mu+ \delta}(x) >  \eta_{\mu}(x),
\eta_{\mu}(x) = k \Big ) & = 
\boldsymbol{\mathcal{P}} \Big ( X_x \in [\nu_{\leq k-1}(\mu), \nu_{\leq k}(\mu)) \cap 
[\nu_{\leq k}(\mu+\delta), 1 ]  \Big ) \\
& =  \nu_{>k}(\mu + \delta) - 
  \nu_{>k}(\mu )  =  \delta \, \,  \nu'_{ > k} (\mu)  + E_{k, \delta},
\end{split},
\end{align}
where the last  identity defines  $E_{k, \delta}$,   which satisfies $|E_{k, \delta}| \leq   \delta^2 \, \, \frac{1}{\min\{\mu,1\}} \, k \,  \nu_{>k-1}(\mu+\delta)$ for any $\delta > 0$.
Fix now an arbitrary particle configuration $\tilde \eta \in \mathcal{H}_a$. 
Since  $\lim_{\delta \rightarrow 0^+} \boldsymbol{\mathcal{P}}
\big (  \mathcal{E}^{x,-} \, \big ) = 1,$
we deduce  from the definition of particle-essential pair, from  Remark \ref{remark:dependences} and from (\ref{eq:equation2}) that,
\begin{multline}\label{eq:equation1}
 \lim_{\delta \rightarrow 0^+}  \,   \frac{1}{\delta} g_{\delta}( \tilde \eta) \\ =  
  \lim\limits_{\delta \rightarrow 0^+}  \sum\limits_{x \in K}
\boldsymbol{\mathcal{P}} \Big (  (\eta_{\mu + \delta}, \tau_\lambda) \in \mathcal{A},  
(\eta_{\mu }, \tau_\lambda) \not\in \mathcal{A}, \,   \mathcal{E}^{x,-}, \,  \eta_\mu \in \mathcal{E}^x_{\tilde \eta} \, \Big | \, \,\mathcal{E}^{x,+}, \eta_\mu(x) = \tilde \eta(x)  \Big )  \nu^\prime_{> \tilde \eta(x)}  \\ = \sum\limits_{x \in K} \boldsymbol{\mathcal{P}} \Big ( \Big  \{ (\eta_\mu, \tau_\lambda) \in \{ (x, \tilde \eta(x) )  \mbox{ is p-essential for }
\mathcal{A} \big \}  \Big \} \cap \{ \eta_\mu \in   \mathcal{E}^{x}_{\tilde \eta} \}  \Big )\nu^\prime_{> \tilde \eta(x)}.
\end{multline}
Moreover, note that from (\ref{eq:equation2}),  from the fact that
$\nu^\prime_{>k} = \nu_{k}$ for any $k \in \mathbb{N}_0$,
and from  Remark \ref{remark:dependences}, we deduce that
for any $\delta> 0$,
\begin{align*}
\frac{1}{\delta} g_{\delta}(\tilde \eta) &  \leq  
\sum\limits_{x \in K} \boldsymbol{\mathcal{P}}( \eta_\mu \in \mathcal{E}^x_{\tilde \eta})  \, \big (  \nu_{\tilde \eta(x)}  + \frac{E_{k,\delta}}{\delta} \big ) \\ & \leq  |K|  \boldsymbol{\mathcal{P}}( \eta_\mu = \tilde \eta )  \, + \, \delta
 \frac{1}{\min\{\mu, 1\}}  \boldsymbol{\mathcal{P}}( \eta_\mu = \tilde \eta )  \sum_{x \in K}   \, \tilde \eta(x)  \,  \frac{\nu_{> \tilde \eta(x) -1}(\mu+\delta)}{\nu_{\tilde \eta(x) }(\mu)}.
\end{align*}
Since the quantity in the RHS 
is summable in $\tilde \eta$ and the sum is uniformly bounded for $\delta \in (0, 1)$,
we can use the dominated convergence theorem  and deduce from (\ref{eq:equation1}) that,
\begin{multline*}
\lim_{\delta \rightarrow 0^+}  \, 
\sum\limits_{ \tilde \eta \in \mathcal{H}_a  }   \frac{1}{\delta} g_{\delta}( \tilde \eta) =
\sum\limits_{ \tilde \eta \in \mathcal{H}_a    } \lim_{\delta \rightarrow 0^+}  \,   \frac{1}{\delta} g_{\delta}( \tilde \eta) \\ =\sum\limits_{ \tilde \eta \in \mathcal{H}_a    }  \sum\limits_{x \in K} \boldsymbol{\mathcal{P}} \Big ( \Big  \{ (\eta_\mu, \tau_\lambda) \in \{ (x, \tilde \eta(x) )  \mbox{ is p-essential for }
\mathcal{A} \big \}  \Big \} \cap \{ \eta_\mu \in   \mathcal{E}^{x}_{\tilde \eta} \}  \Big )\nu^\prime_{> \tilde \eta(x)}  \\ = \sum\limits_{x \in K, j \in \mathbb{N}_0}  \mathcal{P}_{\lambda, \mu} \big (  (x,j)\, \mbox{ is particle-essential for } \mathcal{A} \big )  \, \,   \nu'_{ > j}(\mu).
\end{multline*}
To conclude the proof it remains then to show that the  second term in the RHS of  (\ref{eq:expressionmu})
equals zero.

\subsubsection{Second term in the RHS  of  (\ref{eq:expressionmu})}
We now show that the second term in the RHS of  (\ref{eq:expressionmu}) equals zero. 
For this,  we apply the union bound and obtain that,
\begin{multline}
\label{eq:equation3}
\boldsymbol{\mathcal{P}} \Big (  (\eta_{\mu + \delta}, \tau_\lambda) \in \mathcal{A},  
(\eta_{\mu }, \tau_\lambda) \not\in \mathcal{A}, \mathcal{E}^2 \Big ) \\   \leq  
\sum\limits_{ \substack{x_1, x_2 \in K  \\ x_1 \neq x_2}} 
\sum\limits_{ k_1, k_2 \geq 0} 
\boldsymbol{\mathcal{P}} \Big (  
\eta_{\mu}(x_1) = k_1,  
\eta_{\mu}(x_2)=k_2,
 \eta_{\mu+ \delta}(x_1) > k_1,   
 \eta_{\mu+ \delta}(x_2) > k_2  \Big )   \\
\leq  |K|^2 
\Big (  \sum\limits_{k \geq 0} 
\boldsymbol{\mathcal{P}} \Big (  
\eta_{\mu}(o) = k, 
 \eta_{\mu+ \delta}(o) > k \Big )  
 \Big )^2
 \\ \leq 
 |K|^2 \Big (   \sum\limits_{k \geq 0}   \big (  \nu_{>k}(\mu + \delta) - 
  \nu_{>k}(\mu ) \big )    \Big )^2
 = 
 \delta^2    \, |K|^2,
 \end{multline}
 where we used the fact that for the Poisson distribution $\sum_{k \geq 0} \nu_{>k}(\mu) = \mu$.
The previous inequality implies that the second term in the RHS of (\ref{eq:expressionmu}) equals zero.
This concludes the proof of  the right partial derivative.

\subsubsection{The left partial derivative} 
We now show that (\ref{eq:partial2}) holds with `$\delta \rightarrow 0^-$'
in place of `$\delta \rightarrow 0$'. 
Since $\mathcal{A}$ is increasing, we obtain that, for positive $\delta>0$, 
$$
\mathcal{P}_{\lambda, \mu - \delta}(\mathcal{A}) - 
\mathcal{P}_{\lambda, \mu}(\mathcal{A}) = -  \sum\limits_{\tilde \eta \in \mathcal{H}_a} \boldsymbol{\mathcal{P}} \big (  (\eta_{\mu - \delta}, \tau_\lambda) \not\in \mathcal{A},  
(\eta_{\mu }, \tau_\lambda) \in \mathcal{A}, \eta_\mu = \tilde \eta \big ).
$$
The proof is now analogous to that of the right partial derivative,  with the sets 
\begin{equation*}
\begin{split}
 & \mathcal{D}^{x,+}   := \{ \eta_{\mu- \delta}(x) < \eta_\mu(x)   \} , \\
 & \mathcal{D}^{x,-}  : = \{\forall y \in K \setminus \{x\},  
\eta_{\mu - \delta}(y) = \eta_\mu(y)    \}, \\
& \mathcal{D}^{2}   : = \{ \exists x_1, x_2 \in K \, \, : \, \, 
x_1 \neq x_2 \mbox{ and } \eta_{\mu - \delta}(x_1)  < \eta_\mu(x_1),  \,  \eta_{\mu- \delta}(x_2) < \eta_\mu(x_2) \,  \}, \\
\end{split}
\end{equation*}
playing the role of those defined in (\ref{eq:setsE}), with
\begin{align}\label{eq:equation2inversa}
\begin{split}
\boldsymbol{\mathcal{P}} \Big (\eta_{\mu -  \delta}(x) <  \eta_{\mu}(x),
\eta_{\mu}(x) = k \Big ) & = 
\boldsymbol{\mathcal{P}} \Big ( X_x \in [\nu_{\leq k-1}(\mu), \nu_{\leq k}(\mu)) \cap 
[0,  \nu_{\leq k-1}(\mu -\delta) ]  \Big ) \\
& 
 =   \nu_{ > k - 1 }(\mu ) -     \nu_{ > k - 1}(\mu - \delta)  = 
    \delta \, \,  \nu'_{ > k - 1}  + o(\delta).
\end{split}.
\end{align}
playing the role of (\ref{eq:equation2}) for every $k \geq 1$,
and 
\begin{multline*}
\lim\limits_{\delta \rightarrow 0^+} 
\boldsymbol{\mathcal{P}} \Big (  (\eta_{\mu - \delta}, \tau_\lambda) \not\in \mathcal{A},  
(\eta_{\mu }, \tau_\lambda) \in \mathcal{A}, \,   \mathcal{D}^{x,-}, \,  \eta_\mu \in \mathcal{E}^x_{\tilde \eta} \, \big | \, \,\mathcal{D}^{x,+}, \eta_\mu(x) = \tilde \eta(x)  \Big )  =  \\
 \boldsymbol{\mathcal{P}} \Big (  \Big \{ (\eta_\mu, \tau_\lambda) \in \{ (x, \tilde \eta(x) - 1 )  \mbox{ is p-essential for }
\mathcal{A} \big \}  \Big \}  \cap   \{ \eta_\mu \in  \mathcal{E}^{x}_{\tilde \eta}   \} \Big ).
\end{multline*}
being used  for any $\tilde \eta \in \mathcal{H}_a$
and $x \in K$ such that $\tilde \eta(x) \geq 1$
in the step which is analogous  to (\ref{eq:equation1}).
This concludes the proof.

\section{The key differential inequality}
\label{sect:proofs}
The goal of this section is to state and prove Theorem \ref{theo:first comparison} below, which provides a precise formulation of the differential inequality (\ref{eq:inequality}). 
This section is divided into three subsections. In the first subsection we provide an alternative formula for  (\ref{eq:partial1}), corresponding to
Proposition \ref{prop:alternative form} below.
In the second subsection we state an important comparison lemma. In the last subsection we state and prove our differential inequality.

\subsection{Alternative formula for  (\ref{eq:partial1})}
Our first step is  to provide an alternative formula for the partial derivative with respect to $\lambda$ which appears in Theorem \ref{theo:Russo}.
This is a consequence of Lemma \ref{lemma:essentiality}
and is important for the comparison with the partial derivative with respect 
to $\mu$.
We start with a technical lemma,
which is a consequence of Lemma \ref{lemma:essentiality}.
\begin{lemma}\label{lemma:n=Ly2}
Let $\mathcal{A}$ be a  relevant event with domain $K$.
For every $y \in K$,  $ n \in \mathbb{N}$,
\begin{equation}\label{eq:inclusionempty}
\big \{ (y, n) \mbox{ is s-essential for } \mathcal{A}  \big \} \,  
\cap \, \big \{(\eta, \tau) \in \mathcal{W}  :  \, M_{K, \eta, \tau}(y)  \neq n -1 \mbox{ and } S_{\tau}^{y,n}>0 \big \} = \emptyset.
\end{equation}
\end{lemma}
\begin{proof}
Suppose that   $(\eta, \tau) \in \{  (y,n)$ is s-essential for $\mathcal{A}\} \cap \mathcal{W}$  and that  $S_{\tau}^{y,n}>0$.  We will show that it is necessarily the case that $M_{K,\eta, \tau}(y) = n - 1$, thus implying (\ref{eq:inclusionempty}).
To begin, we  deduce by definition of sleeping-essential pair and by the fact that $S_{\tau}^{y,n}>0$ that,
\begin{equation}\label{eq:desiredcontra}
(\eta, \tau) \not\in \mathcal{A} \quad \mbox{ and } \quad 
(\eta, \Gamma_-^{y,n}(\tau)) \in \mathcal{A}.
\end{equation}
Suppose  that  $M_{K \eta, \tau}(y) \neq n - 1 $. This will lead  to a contradiction.
Indeed, by  Lemma \ref{lemma:essentiality} we deduce that 
$
M_{K, \eta, \tau}    = 
M_{K, \eta, \Gamma_-^{y,n}(\tau)}.
$
From this, from the fact that $\mathcal{A}$  is relevant 
and from the fact that  $(\eta, \tau) \not\in \mathcal{A}$, 
 we deduce that  $(\eta,  \Gamma_-^{y,n}(\tau)) \not\in \mathcal{A}.$ This, however, contradicts (\ref{eq:desiredcontra}), and we obtained the desired contradiction.
\end{proof}

We now present our alternative formula for  the partial derivative with respect to $\lambda$ which appears in Theorem \ref{theo:Russo}.

\begin{proposition}\label{prop:alternative form}
Let $\mathcal{A}$ be a   relevant event. 
for every $\lambda \in (0, \infty)$ we have that,
\begin{multline}
\label{eq:newpartial1}
 \frac{ \partial}{\partial \lambda }  \mathcal{P}_{\lambda, \mu} \big (  \,  \mathcal{A}  \big ) \\
= - \frac{1}{\lambda ( 1+ \lambda)} \, \sum\limits_{y \in K} 
\sum\limits_{n=1}^{\infty} 
 \mathcal{P}_{\lambda, \mu} \Big (   \{ (y,n) \mbox{ is s-essential for } \mathcal{A}\} \,  \cap \, \{ M_{K, \eta, \tau}(y) = n -1  \}  \cap  \{ S^{y,n} > 0 \}   \,   \Big )
\end{multline} 
\end{proposition}
\begin{proof}
Below we use 
Remark \ref{remark:dependences} for the first identity and Lemma \ref{lemma:n=Ly2} for the third identity, obtaining that, for every $y \in K$,
\begin{align*}
& \sum\limits_{j=1}^{\infty} 
\mathcal{P}_{\lambda, \mu}\big (  (y,j) \mbox{ is s-essential} \big )   =
\frac{1+\lambda}{\lambda}
\, \sum\limits_{j=1}^{\infty} 
\mathcal{P}_{\lambda, \mu}\big (   \{(y,j) \mbox{ is s-essential} \big\} \cap \{ S^{y,j} >0\} \big  ) \\
& =  \frac{1+\lambda}{\lambda}  \, \, \, \sum\limits_{j=1}^{\infty} 
\sum\limits_{n=0}^{\infty} 
\mathcal{P}_{\lambda, \mu}\big (   \{(y,j) \mbox{ is s-essential} \big\}   \cap \{ S^{y,j} >0\} \cap \{  M_{K}(y) = n \} \big  ) \\
& = \frac{1+\lambda}{\lambda}  \, \, \, \sum\limits_{j=1}^{\infty} 
\mathcal{P}_{\lambda, \mu}\big (   \{(y,j) \mbox{ is s-essential} \big\}  \cap  \{ S^{y,j} >0\} \cap \{  M_{K}(y) = j-1 \} \big  ).
\end{align*}
By using the previous formula and Theorem \ref{theo:Russo}, we conclude the proof.
\end{proof}

\subsection{Comparison lemma}
We now state and prove our  comparison lemma,
Lemma \ref{lemma:comparisonlemma} below, 
which is the core of the proof of our differential inequality. 
 \begin{lemma}[Comparison lemma]\label{lemma:comparisonlemma}
Consider any relevant event $\mathcal{A}$ with domain $K \subset V$. 
For any $y \in K$,  $n, j \in \mathbb{N}_0$,   with $n>0$, we have that,
\begin{equation}\label{eq:inclusion}
\begin{split}
&  \mathcal{W} \cap \big \{
(y, n ) \mbox{ is s-essential for } \mathcal{A} \, \big \}
\cap \big \{ \eta(y) = j   \big \} \cap  \big \{ S_{\tau}^{y,n}>0 \big \} \cap \{  M_{K, \eta, \tau}(y)= n - 1  \} \\
\subset \,  &  \, \mathcal{W} \cap   \big  \{(y, j)   \mbox{ is p-essential for $\mathcal{A}$}
\big \} \cap \big \{  
\eta(y) = j   \big \} \cap \big \{  S_{\tau}^{y,n}>0 \big \} \cap 
 \{ M_{K, \eta, \tau}(y)  = n -1  \}.
 \end{split}
\end{equation}
\end{lemma}
\begin{proof}
Let  $(\eta,  \tau)$ be a realisation which belongs to the event in the left-hand side (LHS) of (\ref{eq:inclusion}), we will show that it also belongs to the event in the RHS.
Set  $n = M_{K, \eta, \tau}(y)  + 1$.
By definition of sleeping-essential pair and by the fact that $S^{y,n}_{\tau}>0$ we deduce that,
\begin{equation}\label{eq:dom0}
(\eta,  \tau ) \not\in \mathcal{A} \quad  \mbox{ and } \quad (\eta,  \Gamma_-^{y,  n} (\tau)) \in \mathcal{A}.
\end{equation}
By monotonicity,  Lemma \ref{lemma:monotonicity},
we have that,
\begin{equation}\label{eq:dom1}
M_{K,  \eta, \tau} \leq  M_{K,  \eta, \Gamma_-^{y,  n}(\tau)} \leq M_{K, \eta^y,  \Gamma_-^{y, n}(\tau)}.
\end{equation}
By the fact that, by the Abelian property, 
   $n =  M_{K, \eta, \tau}(y) + 1 < M_{K, \eta^y, \tau}(y)  +1$,
   we deduce from  Lemma \ref{lemma:essentiality} that
\begin{equation}\label{eq:dom21}
M_{K,  \eta^y,  \Gamma_-^{y, n}(\tau)} =
M_{K, \eta^y, \tau}.
\end{equation}
From (\ref{eq:dom0}), (\ref{eq:dom1})  (\ref{eq:dom21}), and from the fact that  $\mathcal{A}$
is relevant we deduce that, 
$
(\eta^y,  \tau \, ) \in \mathcal{A}.
$
Summarising, $(\eta, \tau)$ is such that
  \textbf{(1)} $(\eta, \tau) \not\in \mathcal{A}$ by (\ref{eq:dom0}),
\textbf{(2)}  $\eta(y) = j$ by assumption, and
\textbf{(3)} $(\eta^y,  \tau \, ) = (\eta^{y,j+1}, \tau) \in \mathcal{A}$
by (\ref{eq:dom0}), and  (\ref{eq:dom21}).
These three facts imply that  the pair  $(\eta, \tau)$  belongs to the event `$(y,j)$ is p-essential for $\mathcal{A}$'.
From this we deduce that $(\eta, \tau)$ belongs to the event in the RHS of (\ref{eq:inclusion}). This concludes the proof.
\end{proof}

\subsection{Differential inequality}
We are now ready to state the main result of this section. Its
proof will employ our comparison lemma. 
\begin{theorem}[{Differential inequality}]\label{theo:first comparison}
Let  $G = (V, E)$ be an arbitrary undirected  locally-finite connected graph, let 
$\mathcal{A}$ be a relevant event.
Then, for any $(\lambda, \mu) \in \mathbb{R}_+^2$, we have that,
\begin{equation}\label{eq:relation derivative 1}
 \,  -\frac{ \partial}{\partial \lambda }   \mathcal{P}_{\lambda, \mu} \big (  \, \mathcal{A} \,  \big )
\leq \, \, \frac{1}{\lambda(1+\lambda)} \, \frac{ \partial}{\partial \mu } \, \mathcal{P}_{\lambda, \mu} \big (  \, \mathcal{A} \,  \big ).
\end{equation} 
\end{theorem}
\begin{proof}
Let $K \subset V$ be the domain of the relevant event $\mathcal{A}$. 
Using Proposition \ref{prop:alternative form} for the first step,  our Comparison Lemma  for the second step, Remark \ref{remark:dependences}, the fact that for Poisson distributions $\nu^{\prime}_{>k} = \nu_k$ and  Theorem \ref{theo:Russo} for the  last step, we obtain that,
\begin{align*}
& -\frac{ \partial}{\partial \lambda }  \mathcal{P}_{\lambda, \mu} \big (  \, \mathcal{A} \,  \big ) \\
& =   
\frac{1}{\lambda(1+ \lambda)} \, \sum\limits_{y \in K}  \sum\limits_{\substack{j \geq 0 \\ n \geq  1}}
 \mathcal{P}_{\lambda, \mu}\Big ( \{ (y,n) \mbox{ is s-ess.}\}  \cap \{ \eta(y) = j\}  \cap  \{ S_\tau^{y,n} > 0 \}  \cap \{   M_{K, \eta, \tau}(y) = n - 1  \} \Big ) \\
& \leq  \frac{1}{\lambda(1+ \lambda)} \, \sum\limits_{y \in K}  \sum\limits_{\substack{j \geq 0  \\ n \geq 1}}
 \mathcal{P}_{\lambda, \mu}\Big ( \{ (y,j) \mbox{ is p-ess.}\}  \cap \{ \eta(y) = j\}  \cap  \{ S_\tau^{y,n} > 0 \}  \cap \{  M_{K, \eta, \tau}(y)  = n - 1  \} \Big ) \\
  & \leq 
 \frac{1}{\lambda ( 1+ \lambda)} \, \sum\limits_{y \in K} 
 \sum\limits_{\substack{j=0}}^\infty
 \mathcal{P}_{\lambda, \mu} \big ( \{(y,j) \mbox{ is p-essential} \} \cap \{ \eta(y)=j\}\big )  \\
& =   \frac{1}{\lambda (1+ \lambda)} \frac{ \partial}{\partial \mu }  \mathcal{P}_{\lambda, \mu} \big (  \,  \mathcal{A} \,  \big ).
 \end{align*}
This concludes the proof.
\end{proof}

\section{Proof of Theorems \ref{theo:continuity} and \ref{theo:strict monotonicity}}
\label{sect:theoremsproof}
Recall the definition of the curve $ \mathcal{C}_{\lambda, \mu}$ 
which was provided in (\ref{eq:curve}).
We will start with the proof of Theorem \ref{theo:strict monotonicity}.

\begin{proof}[{Proof of Theorem \ref{theo:strict monotonicity}}]
Consider any  relevant event $\mathcal{A}$.
Let $(\lambda, \mu) \in \mathbb{R}_+^2$ be an arbitrary point in the phase diagram, 
take any arbitrary point
$
(x,y ) \in  {\mathcal{C}}_{\lambda, \mu},
$
we assume that $x > \lambda$ (when $x = \lambda$, the result we are going to prove is already known from \cite{Rolla}).
Let 
$
( X(t), Y(t)   )_{ t \in [\lambda, \infty)   }
$
be a curve such that $X(0) : = \lambda$, $Y(0) : = \mu$, 
and for any $t \in [\lambda, \infty)$,
\begin{align*}
\begin{cases}
X(t) &: = t,  \\
Y(t) & : = s ( t - \lambda) + \mu,
\end{cases}
\end{align*}
where $s \geq \frac{1}{\lambda(1+\lambda)}$ is such that there exists a positive $T \in \mathbb{R}$ such that $X(T) = x$, $Y(T) = y$.  
From the fundamental theorem of calculus we deduce that,
\begin{align*}
\mathcal{P}_{x,y}(\mathcal{A}) & = 
\mathcal{P}_{\lambda, \mu}( \mathcal{A}) + 
\int_{\lambda}^{x} dt \,  \, \, \nabla \mathcal{P}_{X(t), Y(t)}( \mathcal{A}) \, \cdot \big (\partial_t X(t), \partial_t Y(t)\big ) =  \\
& = 
\mathcal{P}_{\lambda, \mu}( \mathcal{A}) + 
\int_{\lambda}^{x} dt \,  \, \Big ( \, 
\partial_{\lambda} \mathcal{P}_{\lambda, \mu}(\mathcal{A}) \, \big \vert_{ \lambda = t,  \mu = Y(t)} \, + \, \, s \, \, 
\partial_{\mu} \mathcal{P}_{\lambda, \, \mu}(\mathcal{A})  \big  \vert_{\lambda = t,  \mu = Y(t)   } 
\Big )  \\
& \geq 
\mathcal{P}_{\lambda, \mu}( \mathcal{A}) + 
\int_{\lambda}^{x} dt \, \, 0 \\ &  \geq  \mathcal{P}_{\lambda, \mu}( \mathcal{A}),
\end{align*}
where for the first inequality we used Theorem \ref{theo:first comparison} and the fact that $s \geq \frac{1}{\lambda(1+\lambda)} \geq  \frac{1}{t(1+t)}$ for every $t \geq 0$.
This concludes the proof.
\end{proof}
We now provide a more general definition of \textit{critical density,}
\begin{equation}
\label{eq:generaldefinitiondensity}
\forall \lambda \in [0, \infty) \quad \quad \zeta_c(\lambda) : =  \inf \Big \{ \mu \in \mathbb{R}^+_0  \, \, : \, \,  \mathcal{P}_{\lambda, \mu}( m(o) = \infty  ) > 0  \Big \},
\end{equation}
By monotonicity, the random variable $m_{\eta, \tau}(o)$ is well defined for every  locally-finite connected  infinite graph $G$.
From now on we will say that ARW \textit{fixates} if $ m(o) < \infty$ and that it is \textit{active} otherwise.  Such a notion of activity and fixation reduces to the one which has been introduced in Section \ref{sect:Intro},
thus implying   the identity $\zeta_c(\lambda) = \mu_c(\lambda)$,
whenever  Lemma \ref{lemma:01law} holds.
The next theorem is an immediate consequence of 
Theorem \ref{theo:strict monotonicity}.
\begin{theorem}\label{prop:ifactivethenactive}
Let $G$ be a locally-finite connected  graph.  For any $(\lambda, \mu) \in \mathbb{R}_+^2$ and 
$(\lambda^{\prime},\mu^{\prime}) \in \mathcal{C}_{\lambda, \mu}$ we have that,
$$
\mathcal{P}_{\lambda, \mu}(\mbox{ARW active} )  \leq 
\mathcal{P}_{\lambda^\prime, \mu^\prime}(\mbox{ARW active}  ).
$$
\end{theorem}
\begin{proof}
Define $B_L : = \{x \in V \, : \, d(x,o) \leq L\}$.
For any $L, H \in \mathbb{N}$ consider the relevant event,
$
\mathcal{A}_{L, H} :=    \{  M_{B_L}(o) > H   \}.
$
From our  monotonicity theorem, Theorem \ref{theo:strict monotonicity}, we deduce that 
for any $L, H \in \mathbb{N}$,
$
\mathcal{P}_{\lambda, \mu}(\mathcal{A}_{L,H} ) \leq 
\mathcal{P}_{\lambda^\prime, \mu^\prime}(\mathcal{A}_{L,H} ).
$
From this and from Lemma \ref{lemma:01law} we then deduce that,
\begin{multline*}
\mathcal{P}_{\lambda, \mu}(\mbox{ARW active} ) =
\lim\limits_{ H \rightarrow \infty  }
\lim\limits_{ L \rightarrow \infty  }
 \mathcal{P}_{\lambda, \mu}(\mathcal{A}_{L,H} )  \\
  \leq 
\lim\limits_{ H \rightarrow \infty  }
\lim\limits_{ L \rightarrow \infty  }
\mathcal{P}_{\lambda^\prime, \mu^\prime}(\mathcal{A}_{L,H} )
=
\mathcal{P}_{\lambda^\prime, \mu^\prime}(\mbox{ARW active}  ).
\end{multline*}
This concludes the proof.
\end{proof}

\subsection{Proof of Theorem  \ref{theo:continuity}}
We now present the proof of our continuity theorem. 
\begin{proof}
It suffices to prove the second claim of Theorem \ref{theo:continuity},
since it implies the first claim.
In the whole proof we use the fact that
it is known from
\cite{Shellef} that  on any vertex-transitive graph the critical density is finite for every $\lambda \in [0, \infty)$.
Consider an arbitrary $\lambda \in (0, \infty)$ and, for any $\epsilon > 0$ and $t \in (- \infty, \infty)$ define the function,
$$
Y_{\epsilon}(t) := \zeta_c(\lambda) +  \epsilon +  \frac{1}{\lambda(1+\lambda)}   \, (t - \lambda).
$$
Suppose that,
\begin{equation}\label{eq:lookcontra1}
\limsup\limits_{\delta \rightarrow 0^+} \frac{\zeta_c(\lambda + \delta) - \zeta_c(\lambda)}{\delta}   >  \frac{1}{\lambda (1+\lambda)},
\end{equation}
we look for a contradiction with this claim.
From (\ref{eq:lookcontra1}) it follows that  we can find a small enough $\Delta > 0$ such that
there exists an infinite  positive sequence  $(\delta_n)_{n\in \mathbb{N}}$ converging to zero with $n$ such that, for any large enough $n$,
\begin{equation}\label{eq:contradict1}
 \zeta_c(\lambda + \delta_n) \geq  \zeta_c(\lambda) \, + \, 
\delta_n (\frac{1}{\lambda(1+\lambda)} + \Delta).
\end{equation}
From Theorem \ref{prop:ifactivethenactive} and from the definition of the critical density
we deduce that,
$
 \zeta_c(t) \leq Y_{\epsilon}(t).
$
From this we deduce that,
$$
\forall \epsilon > 0, \quad \forall n \in \mathbb{N}, \quad   \zeta_c(\lambda + \delta_n)  \leq 
 Y_{\epsilon}(\lambda + \delta_n) = \zeta_c(\lambda)  + \, \epsilon  \, +  \,  \delta_n   \, 
\frac{1}{\lambda(1+\lambda)}.
$$
We can then find $n \in \mathbb{N}$ large and $\epsilon >0 $ small such that the previous inequality contradicts (\ref{eq:contradict1}) and thus also 
 (\ref{eq:lookcontra1}). We then found the desired contradiction and concluded the proof of the second claim of the theorem for $\delta \rightarrow 0^+$. 
 The proof for $\delta \rightarrow 0^-$ is analogous, hence the proof is concluded.

\end{proof}

We conclude with a remark about the generality of our results,
which hold on  graphs more general than unimodular.

\begin{remark}\label{rem:generality}
Relevant events do not depend on the clock realisations of the continuous time dynamics and have been defined in the framework of the Diaconis-Fulton representation, which is well-defined on any 
locally-finite  infinite connected  graph.  Hence,   our Theorem \ref{theo:strict monotonicity} can be stated in wider generality,  namely for any infinite connected locally-finite graph, provided that the probability measure $\mathbb{P}_{\lambda, \mu}$ is replaced by $\mathcal{P}_{\lambda, \mu}$ in 
(\ref{eq:statementmonotonicity}).

Moreover, contrary to  (\ref{eq:criticaldensity}), the  critical density
(\ref{eq:generaldefinitiondensity}) is well-defined on   any locally-finite  infinite connected  graph. 
Our proof of Theorem \ref{theo:continuity} then
 implies that $\zeta_c(\lambda)$
is a continuous function of $\lambda$ in $(0, \infty)$ on any locally-finite infinite connected graph,
provided that it is known that $\zeta_c(\lambda)< \infty$ for some $\lambda > 0$. 
\end{remark}

\section*{Acknowledgements}
This work started as the author was affiliated to Technische Universit\"at Darmstadt, it has been  carried on while the author was affiliated to the University of Bath, it was concluded as the author was affiliated to the Weierstrass Institute, Berlin,
it was revised as the author was affiliated to Sapienza Universit\`a di Roma. 
The author acknowledges support from DFG German Research Foundation BE/5267/1
and from  EPSRC Early Career Fellowship EP/N004566/1.
The author thanks the two anonymous referees for carefully reviewing the paper and their important and useful comments.    

\end{document}